\documentclass[11pt]{amsart}
\usepackage{graphicx}
\usepackage{colortbl}
\usepackage{amsmath}
\usepackage{amssymb, amsfonts, amsmath, amsthm, amscd}
\usepackage{mathrsfs}
\usepackage{color, graphics}
\usepackage[all]{xy}
\usepackage{hyperref}
\usepackage{pstricks,pst-plot}
\makeatletter
\renewcommand{\@secnumfont}{\bfseries}
\renewcommand{\section}{\@startsection{section}{1}%
  \z@{.7\linespacing\@plus\linespacing}{.5\linespacing}%
  {\normalfont\bfseries\centering}}
\makeatother
 \calclayout
\theoremstyle{plain}
\newtheorem{Theorem}{Theorem}[section]
\newtheorem{Lemma}[Theorem]{Lemma}
\newtheorem{Proposition}[Theorem]{Proposition}
\newtheorem{Corollary}[Theorem]{Corollary}
\newtheorem{Remark}[Theorem]{Remark}
\newtheorem{Definition}[Theorem]{Definition}
\newtheorem{Example}[Theorem]{Example}
\newtheorem{Conjecture}[Theorem]{Conjecture}

\newcommand{\gon}{\hbox{\rm gon}}
\newcommand{\cli}{\hbox{\rm Cliff}}

\def\gon{\mbox{gon}}
\def\cli{\mbox{Cliff}}
\def\deg{\mbox{deg}}
\def\min{\mbox{min}}

\def\P{\mathbb P}
\def\Z{\mathbf Z}

\def\co{\mathcal O}

\def\cl{\mathcal L}
\def\ck{\mathcal K}
\def\cm{\mathcal M}

\def\ce{\mathcal E}
\def\cn{\mathcal N}

\def\Z{\mathbf Z}%

\def\gon{\hbox{\rm gon}}
\def\cli{\hbox{\rm Cliff}}

\def\gcd{\hbox{\rm gcd}}

\begin{document}

%\title[Explicit presentations and secant spaces of nonspecial line bundles ]{%
%    Explicit presentations and secant spaces\\ of nonspecial line bundles  on algebraic curves }

%\title[Secant spaces of nonspecial line bundles ]{%
%    Secant spaces of nonspecial line bundles \\on curves and applications to szygies }
%\title[Secant spaces and  syzygies of projective curves ]{%
%    Secant spaces and syzygies of curves\\ embedded by nonspecial line bundles }

\title[ Explicit presentations of nonspecial line bundles  and  secant spaces ]
{%
 Explicit presentations of nonspecial line bundles  and  secant spaces} 
%  A study of  nonspecial line bundles in terms of secant spaces and its applications to syzygies}

%\title[Explicit presentations and secant spaces of nonspecial line bundles ]{%
%    Explicit presentations and secant spaces\\ of nonspecial line bundles  on algebraic curves }

\author{Seonja Kim}
\address{Department of  Electronic Engineering,
Chungwoon University, Hongseong-gun, Chungnam, 350-701, Korea}
\email{sjkim@chungwoon.ac.kr}

\thanks{This research was supported by Basic Science Research Program through the National Research Foundation of Korea(NRF) funded by the Ministry of Education, Science and Technology (2011-0011224).}


\begin{abstract}
%For arbitrary pair $( D, E )$   of effective divisors   on  a smooth curve $X$  with $\gcd (D, E)=0$,  we obtain  a   line bundle $\cl\simeq \ck _X -D +E $, which is nonspecial if  $h^0 (X ,\co _X (D))=1$ and $E>0$. 
%Any  nonspecial line bundle $\cl$ on a smooth curve $X$ admits    presentations $\cl\simeq\ck _X -D +E $  for  $D\ge 0,~ E>0$ with $\gcd (D, E)=0$ and $h^0 (X ,\co _X (D))=1$.  
A line bundle $\cl$ on a smooth curve $X$ is nonspecial if and only if   $\cl$  admits a   presentation $\cl\simeq\ck _X -D +E $  for  some divisors $D\ge 0,~ E>0$ on $X$  with $\gcd (D, E)=0$ and $h^0 (X ,\co _X (D))=1$.
In this work, we define a minimal presentation of $\cl$ which is minimal with respect to $\deg E$  among the presentations.
%A presentation of $\cl$ is said to be minimal if it is minimal with respect to $\deg E$  among the presentations of $\cl$.
  If $\cl\simeq\ck _X -D +E $ with $\deg E \ge 3$ is  a minimal, then $\cl$ is very ample 
and any $q$-points of $\varphi _{\cl} (X)$ with $q \le \deg E -1$ are in  general position but the points of $\varphi _{\cl} (E)$ are not.
We investigate sufficient conditions on divisors $D,E$ for   $\cl\simeq\ck _X -D +E $ to be minimal. Through this,  for a number  $n $ in some range,  it is possible to construct a nonspecial very ample line bundle $\cl\simeq\ck _X -D +E $ on $X$  with/without an $n$-secant $(n-2)$-plane of the embedded curve by taking  divisors $D,E$ on $X$.
%divisors $D,E$ so that  $\cl\simeq\ck _X -D +E $ may be minimal. 
As its applications,
we  construct nonspecial line bundles  which show the sharpness of Green and Lazarsfeld's Conjecture on property  $(N_p)$ 
for general $n$-gonal curves and simple multiple coverings of smooth plane curves  .
\end{abstract}

\maketitle
%%%%%%%%%%%%%%%%%%%%%%%%%%%%%%%%%%%%%%%%%%%%%%%%%%%%%%%%%%%%%%%%%%%%%%%%%%
\section{Introduction}

Throughout this paper, we mean a curve by a reduced irreducible
algebraic curve
 over an algebraically closed field of characteristic zero.
We  will investigate  properties of  nonspecial line bundles $\cl$ on a smooth  curve $X$ with respect to presentations such as  $\cl\simeq \ck _X -D +E $ by using the canonical line bundle $\ck _X$ and effective divisors $D,~ E$  on $X$ with $\gcd (D, E)=0$ and $h^0 (X ,\co _X (D))=1$.
%   introduce a kind of explicit
%description for a nonspecial line bundle $\cl$ on a smooth  curve $X$  such as and investigate some properties of  nonspecial line bundles with respect to such  presentations.

To an  arbitrary pair of effective divisors  $ D$, $ E $   on  $X$  with $\gcd (D, E)=0$  we can associate  a   line bundle $\cl\simeq \ck _X -D +E $, which is nonspecial if  $h^0 (X ,\co _X (D))=1$ and $E>0$. 
Conversely, a nonspecial line bundle $\cl$ on  $X$ also admits an equivalence $\cl\simeq\ck _X -D +E $, which will be called a presentation of $\cl$, for some $D\ge 0,~ E>0$ with $\gcd (D, E)=0$ and $h^0 (X ,\co _X (D))=1$.  However,  a nonspecial line bundle may have several different presentations. Thus we  define  a minimal presentation(:minimal with respect to $\deg E$) as the most efficient one in some sense.

% a  nonspecial line bundle $\cl\simeq \ck _X -D +E $ which will be called  a presentation of $\cl$.

Assume that   $\cl$  is  minimally presented by $\ck _X -D +E $. If $\deg E\ge 3$, then $\cl$ is very ample and any  $q$-points of $\varphi _{\cl} (X)$ with $q \le \deg E -1$ are in  general position but the points of $\varphi _{\cl} (E)$ are not(see Proposition \ref{thm0}).  Accordingly,  nonspecial line bundles can be distinguished by their minimal presentations. Thus finding sufficient conditions for  minimality can be a major issue in this study.  In Section 3, we  explore some sufficient conditions for such presentations to be  minimal   on multiple coverings. Note that every smooth curve is a multiple covering of $\P ^1$.

Now, consider some  details of the brief outline above.
Let $X$ be a smooth curve of genus $g\ge 2$ and $\cl$ be a line bundle
on $X$. If $\cl$ is special(: $h^1 (X,\cl )>0)$, then its residual
line bundle $\ck _X \otimes \cl ^{-1}$ plays a role in investigating
the properties of $\cl$ or $X$ itself, since $\ck _X \otimes \cl
^{-1}$   has global sections and is associated to an effective
divisor.   On the other hand, if $\cl$ is nonspecial(: $h^1 (X,\cl
)=0)$  then the residual line bundle $\ck _X \otimes \cl ^{-1}$ has
no  global sections and hence
no corresponding effective divisors. Accordingly, it is a natural analyzing approach to
express $\ck _X \otimes \cl ^{-1}$ in terms of
effective divisors as follows.

Let $\cl$ be a nonspecial line bundle on a smooth curve $X$ of genus
$g$. Then, there exists a divisor $E>0$ such that
$$h^0 (X,\ck _X \otimes \cl ^{-1}
(E))=1,~~h^0 (X,\ck _X \otimes \cl ^{-1} (E'))=0  \mbox{ for } E' <E.$$
 Hence we have  an effective divisor $D$
 such that
 $$\ck _X \otimes \cl ^{-1}(E) \simeq \co _X  (D)
 \mbox{, equivalently, } \cl\simeq \ck _X(-D+E),$$
with $h^0 (X,\co _X  (D))=1$ and $\gcd (D,E)=0.$

 For an efficient description, we will use some
notations as the
following:\\
 \noindent{$(i)~ g^0_d$}: an effective divisor of degree $d$ with  $h^0 (X, \co _X
 (g^0_d))=1$,\\
 \noindent{$(ii)~(D, E)$}: the greatest common divisor of divisors $D$ and
 $E$,\\
\noindent{$(iii)~ \mathcal L-g^n_d$}: the line bundle $\mathcal
L(-D)$ where $|D|= g^n_d$.
%\noindent{$(iv)$} $ \langle D \rangle _{\cl}:=\cap \{ H \in H^0 (\mathbb P, \co_{\mathbb P} (1))~| ~ H.\varphi _{\cl} (X) \ge D \}$, where $\cl$ is a very ample line bundle with $\P H^0 (X, \cl)^* =: \P $ and $D$ is an effective divisor on $X$.

Using these, a nonspecial line bundle
$\cl$ can be written as
 $$\cl\simeq \ck _X -g^0_d+E$$
  for some divisors $g^0_d$ and $E >0$  on $X$ with  $ (g^0_d ,E)=0$.
Conversely, if  $\cl\simeq \ck _X -g^0_d+E$ for such $g^0_d$ and $E>0$ then $\cl$ is  nonspecial.

\begin{Definition} Let $\cl$ be a nonspecial line bundle on $X$.
(1)  If $\cl \simeq \ck _X -g^0_d +E$ with $ (g^0_d, E )=0$ and $E>0$ then the equivalence is called  a presentation of  type $(d,e)$, where $e:=\deg E$.
 (2) $\cl \simeq \ck _X -g^0_d +E$
  is said to be minimal if any  presentation $\cl \simeq \ck _X -g^0_t+F$ satisfies $\deg F\ge \deg E$. (3) A presentation of type $(0,e)$, i.e., $\cl \simeq \ck _X  +E$ is said to be trivial.
\end{Definition}

Assume that  a nonspecial line bundle $\cl$ is presented by $\ck _X -g^0_d+E$. Then we have the
equality  $h^0 (X,\cl )-h^0 (X,\cl (-E ) ) = \deg E-1$.
 Accordingly,  $\cl$ is
 not globally generated if $\deg E =1$, and  $\cl$ is not very ample if $\deg E=2$.
 Hence  it is a natural question whether a nonspecial line bundle $\cl\simeq\ck _X -g^0_d+E_3$  with $\deg E_3 
=3$ is very ample or not.
 If  the presentation $\ck _X -g^0_d+E_3$  is not minimal, equavelently, $\cl$ has
another presentation $\cl\simeq \ck _X -g^0_{\le d-1}+E'$ with $\deg
 E ' \le 2$, then $\cl$ is not very ample.
% This is just a minimal presentation problem, since  the line bundle $\cl$ is very ample if and only if   $\cl\simeq\ck _X -g^0_d+E_3$ is a minimal presentation. 
%Note that if  $\cl\simeq\ck _X -g^0_d+E_3$ is very ample then the curve $\varphi _{\cl} (X)$ admits a trisecant line spanned by $\varphi _{\cl} (E_3)$. Thus we also raise a question whether  $\varphi _{\cl} (X)$ has no trisecant line in case $\cl\simeq\ck _X -g^0_d+E_4$.
% This is also equivalent to the question on the minimality of $\cl\simeq\ck _X -g^0_d+E_4$.  Such a question can extended to any $\deg E \ge 3$. 
%% If  the presentation $\ck _X -g^0_d+E_3$  is not minimal, equavelently, $\cl$ has
%%another presentation $\cl\simeq \ck _X -g^0_{\le d-1}+E'$ with $\deg
%% E ' \le 2$, then $\cl$ is not very ample. In fact, the very ampleness of   $\ck _X -g^0_d+E_3$ holds
%%if there are no $P,~Q \in X$ and  $g^0_t$ such that $g^0_t+E_3
%%\simeq g^0_d +P+Q$ and $g^0_t+E_3 \neq g^0_d +P+Q$ as
%%divisors, which implies $h^0 (X, \co _X  (g^0_d +P+Q))\ge 2$(see
%%Corollary \ref{coro1}).
%On the one hand,  a presentation $\cl \simeq \ck _X -g^0_d +E$ is minimal
% if there are no effective divisors
% $g^0_t$ and $F$ on $X$  with $\deg F< \deg E$ and $ (g^0_t ,~F )=0$  satisfying
%$g^0_t + E \neq g^0_d +F$ and $g^0_t + E \simeq g^0_d +F$, whence
%   $h^0 (X, \co _X  (g^0_d +F ))\ge 2$(see Theorem \ref{lem1}.).

Likewise, a given nonspecial line bundle $\cl$  may admit various presentions. Here,  the degrees of $g^0_d$ and $ E$ as well as the divisors $g^0_d$ and $E$  depend on
presentations of $\cl$.
%Note that for presentations of  a given nonspecial line bundle $\cl$ the degrees of $g^0_d$ and $ E$
%as well as the divisors $g^0_d$ and $E$  depend on
%presentations of $\cl$.
However, a special line bundle $\cl$ can be written as $\ck _X -D$ for  $D \in |~\ck_X \otimes \cl ^{-1}~|$
 which is  unique up to linear equivalence.
 Thus we would be naturally interested in a
minimal presentation and its uniqueness.

Assume that   a  nonspecial line bundle $\cl$   is  minimally presented by  $ \ck _X -g^0_d +E$ with $\deg E \ge 3$.  Then, by the Riemann-Roch Theorem $\cl$ is very ample and the embedded curve $\varphi_{\cl} (X)$ admits a $\deg E$-secant  $(\deg E -2)$-plane  $\langle E \rangle _{\cl}$ but does not admit $n$-secant $(n-2)$-planes for any $n\le \deg E-1$.
Here,  $ \langle E \rangle _{\cl}:=\cap \{ H \in H^0 (\mathbb P, \co_{\mathbb P} (1))~| ~ H.\varphi _{\cl} (X) \ge E \}$,  $\P  : =\P H^0 (X, \cl)^*  $.
 Moreover, if  $\cl \simeq \ck _X -g^0_d +E$ is a unique minimal presentation then $\langle E \rangle _{\cl}$ is a unique $\deg E$-secant  $(\deg E -2)$-plane of $\varphi_{\cl} (X)$.

Now, observe the case that  a nonspecial line bundle $\cl$ is trivially presented by $\ck _X  +E$. This $\cl\simeq \ck _X  +E$    is in itself  a minimal presentation and the
family of such presentations of $\cl$ corresponds to the linear
system $|E|$.
Note that the minimal presentations of  nonspecial line bundles $\cl$ with 
 $\deg \cl\ge 3g-2$ are always trivial(see Proposition \ref{thm0}, (vi), whereas 
 every  presentation of a nonspecial globally generated line bundle  $\cl$  with $ \deg \cl \le 2g-1$ is  nontrivial, i.e., $\cl$ is always  presented by  $ \ck _X -g^0_d+E$ with
$g^0_d \neq 0$(see Remark \ref{rmknon}, (ii)).

Assume that $\cl$ admits a nontrivial  presentation $\cl \simeq \ck _X -g^0_d+E$.
Then we may assume that $h^0 (X, \co _X (E))=1$(see Proposition \ref{thm0}, (v)) and hence  denote the divisor $E$ by $\xi^0_e$, $e:=\deg E$.
Accordingly, $\cl$ can be written as
$$\cl\simeq \ck _X -g^0_d+\xi^0_e ~ \mbox{ with }  (g^0_d,
\xi^0_e )=0$$
which is   a better explicit description  than the type
of $\cl \simeq \ck _X -D+E$, since notations such as $g^0_d$,
$\xi^0_e$ include information on degrees and dimensions of $|D|$ and $|E|$.

Conversely, we obtain  a nonspecial line bundle $\cl \simeq \ck _X -g^0_d+E$ on $X$ whenever we take effective  divisors $g^0 _d$ and $E\ge 0$ on $X$ with $(g^0_d,E )=0$.  If we obtain sufficient conditions on $D,E$ for $\cl \simeq \ck _X -g^0_d+E$ to be  minimal, then we can construct  a nonspecial line bundle $\cl \simeq \ck _X -g^0_d+E$ with/without an $n$-secant $(n-2)$-plane of the embedded curve $\varphi _{\cl} (X)$ by taking some divisors $D,E$ on $X$.
% which carry some specific properties with respect to secant spaces.
This study  could also provide a clue  to detect   the family of nonspecial line bundles with specific  properties for such secant spaces.

% By the way, we note another perspective on the   presentations of nonspecial  line bundles  such that the curve $\varphi _{\cl} (X) $ for a globally generated $\cl \simeq \ck _X -g^0_d+E$ is a projection of $\varphi _ {\ck _X +E} (X)$ from $ {\langle g^0_d \rangle}_{\ck _X +E } ,$
Note that for a nonspecial  very ample $\cl \simeq \ck _X -g^0_d+E$  the  curve  $\varphi _{\cl} (X) $ is a projection of $\varphi _ {\ck _X +E} (X)$ from $ {\langle g^0_d \rangle}_{\ck _X +E } ,$   whereas  for a special very ample line bundle $\cl \simeq \ck _X -D$ on a nonhyperelliptic curve $X$ the curve $\varphi _{\cl} (X)$  is a projection of the canonical curve $\varphi _ {\ck _X} (X)$ from $ {\langle D \rangle}_{\ck _X}$.
This gives another perspective on our study  that finding a minimal presentation of a very ample line bundle $\cl$ is equivalent to choosing a minimal effective divisor $E$ satisfying  (1) $\varphi _ {\cl} (X)$ is  a projection of $\varphi _ {\ck _X +E} (X)$, (2)  both $\varphi _ {\cl} (X)$  and $\varphi _ {\ck _X +E} (X)$ possess the same  properties with respect to $n$-secant $(n-2)$-spaces.

% so that  $\varphi _ {\cl} (X)$ may be a projection of $\varphi _ {\ck _X +E} (X)$ from a linear space $ {\langle g^0_d \rangle}_{\ck _X +E}$ and both $\varphi _ {\cl} (X)$  and $\varphi _ {\ck _X +E} (X)$ may possess the same  properties with respect to $n$-secant $(n-2)$-spaces.

%on  the   presentation $\cl \simeq \ck _X -g^0_d+E$ of a nonspecial $\cl$ such that the curve $\varphi _{\cl} (X) $  is a projection of $\varphi _ {\ck _X +E} (X)$ from $ {\langle g^0_d \rangle}_{\ck _X +E } ,$
% whereas   $\varphi _{\cl} (X) $ for a special globally generated line bundle $\cl \simeq \ck _X -D$ is a projection of $\varphi _ {\ck _X} (X)$ from $ {\langle D \rangle}_{\ck _X}$.
%%Specifically, the minimality of   $\cl \simeq \ck _X -g^0_d+E$ means that $\varphi _{\cl} (X) $ preserves the  properties of $\varphi _ {\ck _X +E} (X)$ with respect to $n$-secant $(n-2)$-spaces.
%Thus, finding a minimal presentation of $\cl$ is equivalent to choosing an effective divisor $E$ satisfying that  $\varphi _ {\cl} (X)$ is a projection of $\varphi _ {\ck _X +E} (X)$ from a linear space $ {\langle D \rangle}_{\ck _X +E}$ and possesses the same  properties of $\varphi _ {\cl} (X)$ in terms of $n$-secant $(n-2)$-spaces.

    It is interesting that  every presentation  $ \cl \simeq\ck _X -g^0_d+\xi^0_e$ with $d+e \le gon(X)$( resp. $d+e < gon(X)$ ) is a ( resp. unique ) minimal one(see Theorem \ref{corgon}).  On the other hand, there are  examples  of  $ \cl \simeq \ck _X -g^0_d+\xi^0_e$ with $d+e \ge gon(X) +1$  which are not minimal(see Example \ref{exgon}).

    Furthermore, if $\cl$ admits a presentation $\cl\simeq\ck _X -g^0_d+\xi^0_e$ with $d+e = gon(X)$, then the number of  presentations of $\cl$ with the same type $(d,e)$   is at most one plus   the number of pencils  $g^1_{gon(X)}$ on $X$(see Remark \ref{rmkcan}, (iii)).
This means that  for $e\ge 3$ the number of $e$-secant $(e-2)$-planes of $\varphi _{\cl} (X)$ is at most one plus the number of  pencils  $g^1_{gon(X)}$.
% the order of  very ampleness of  such a line bundle $\cl$(see Definition \ref{defor}) is equal to  $(e-2)$ and . 
In addition,  we show that an m-fold covering $X$ of an elliptic curve with $\gon(X)=2m$ have a line bundle $\cl \simeq \ck _X -g^0_d+\xi^0_e$
 admitting  infinitely many presentations  of the same type $(d,e)$ with $d+e=\gon(X)$(see Example \ref{egel}).

Note that  for a smooth curve $X$ with a well known $g^r_d$ the line bundles  $\cl \simeq \ck _X -g^r_d +\xi ^0 _e$ are very typical nonspecial line bundles on $X$. Thus we investigate   minimal presentations of $\cl \simeq \ck _X -g^r_d +\xi ^0 _e$ on the curve  $X$. To do this, we set $\beta := \mbox{max} \{ \deg (\xi ^0 _e , D) | D \in g^r_d \}$. Then we may expect the minimality of $\cl \simeq \ck _X -g^0_{d-\beta} +\xi ^0 _{e-\beta} $, where $\deg  (\xi ^0 _e , D)=\beta$  for a $D\in g^r_d$,
$g^0_{d-\beta} := g^r_d - (\xi ^0 _e , D)$ and $\xi ^0 _{e-\beta}:=\xi ^0 _e -(\xi ^0 _e , D)$.
Such an expectation holds  under some specific conditions and   there is 
 also an  example where the expectation fails(see Theorem \ref{submin}, Example \ref{subex}).

In section 3, we investigate sufficient conditions for  the minimality of presentations of nonspecial line bundles  on multiple coverings.  For  an   $n$-fold covering $X$   via $\phi :X \to Y$  a  presentation $\cl\simeq \ck_X-g^0_d+\xi^0_e$ with  $d+e \le \mu$ is  minimal  if $\deg (g^0_d , \phi ^* (P)) + \deg (\xi^0_e, \phi ^* (Q))\le n$ for  any $P$, $Q \in Y$,  where  $ \mu := \mbox{min}\{ \deg \cn ~| ~\cn  \mbox{:  globally generated and not composed with }\phi \}$(see Theorem \ref{mainthm}).
Specifically, the number $\mu$ is greater than  $\frac{g+1}{2}$(resp. $\frac{g-n\gamma}{n-1}$) for a general $n$-gonal curve(resp. for a simple $n$-fold covering of a smooth curve of genus $\gamma$). Here,  a multiple covering is said to be {\it  simple} if the covering morphism does not factor through. 
 Note that general $g^0_d$ and  $\xi^0_e$ on $X$ satisfy the condition $\deg (g^0_d , \phi ^* (P)) + \deg (\xi^0_e, \phi ^* (Q))\le n$ for  any $P$, $Q \in Y$.
Thus whenever we take general $g^0_d$ and  $\xi^0_e$ on a multiple covering $X$  with ${e\ge 2}$ and $d+e \le \mu$, we obtain a nonpecial line bundle $\cl\simeq \ck_X-g^0_d+\xi^0_e$ on $X$ which is $(e-2)$-very ample.
 This means that for any positive number $q\le \mu -1$ we can construct $q$-very ample nonspecial line bundles on $X$  with a given degree$\ge 2g-1+2e-\mu$.
 It is also notable that  for  an $n$-fold covering $\phi : X \to \P^1$ the condition such that $\deg (g^0_d , \phi ^* (P)) + \deg (\xi^0_e, \phi ^* (Q))\le n$ for  any $P$, $Q \in \P^1$ is necessary  for $\cl\simeq \ck_X-g^0_d+\xi^0_e$ to be  minimal(see Proposition \ref{thmgon}). 

We also deal with  minimal presentations  of  typical line bundles such as $\cl \simeq \ck  _X  - \phi ^* (g^2_d) +\xi ^0 _{e+2}$ and $\cm\simeq \ck  _X  - \phi ^* (g^1_{d-1}) +\xi ^0 _{e+1}$ on  a
 simple $n$-fold covering $X$ of a smooth plane curve $Y$  via $\phi : X \to Y$(see Theorem \ref{planeth}). 

In section 4, we apply minimal  presentations of nonspecial line bundles to  investigate property $(N_p )$, since $(p+1)$-very ampleness is  very closely connected with property $(N_p )$. M. Green and R. Lazarsfeld showed that a line bundle $\cl$ of degree $2g+p$ on a nonhyperelliptic curve $X$  satisfies $(N_p)$ if and only if  $\varphi _{\cl} (X)$ has no (p+2)-secant p-planes, that is, $\cl$ is $(p+1)$-very ample(see \cite{GL1}, Theorem 2). On the other hand,  it is well known that if a very ample line bundle  $\cl$ on $X$   fails to be $(p+1)$-very ample then $\cl$ does not satisfy  $(N_p )$.

Along this line,  the validity of its converse under the condition $\deg \cl\ge 2g+1+p-2h^1(X,\cl)-\cli(X)$ was conjectured by M. Green and R. Lazarsfeld in \cite{GL}.
It is called  Green-Lazarsfeld's conjecture on $(N_p)$.
In fact, they have shown in the paper that this conjecture holds for $(N_0)$.
Since M. Aprodu demonstrated in \cite{Ap1} that general gonality curves satisfy Green's Conjecture on syzygies of canonical curves(:this validity was remarked after Theorem 2 in \cite{Ap1}), we can easily see that the special line bundles on them satisfy Green-Lazarsfeld's conjecture on $(N_p)$ by Theorem 1 in \cite{CKKw}.
Thus a natural question is  on the existence of a very ample line bundle $\cl$ on $X$ with $\deg \cl = 2g+p-2h^1(X,\cl)-\cli(X)$ which does not satisfy $(N_p)$ even if $\varphi _{\cl} (X)$ does not admit a $(p+2)$-secant $p$-plane. Such a line bundle  will  be called  {\it an  extremal  line bundle for Green-Lazarsfeld's conjecture on $(N_p)$}.

Using theorems on the minimality of presentations in section 3, we  verify that   general $n$-gonal curves and   simple $n$-fold coverings of smooth plane curves  carry  nonspecial extremal  line bundles for Green-Lazarsfeld's conjecture on $(N_p)$(see  Theorem \ref{eggon}, \ref{plNp}).  To do this study, we   compute   the Clifford index of   multiple  coverings of smooth plane curves(see Proposition \ref{clif}).

\section{The presentations  of  nonspecial line bundles}
In this section,  we investigate  properties of presentations of nonspecial line bundles on smooth curves. This study   naturally focuses on the  minimal presentations of  nonspecial line bundles which can be regarded as  efficient ones.  Before going to this observation, we will consider a type of refinement of very ampleness  which  is closely related to  minimal presentation. 

Recall that a  line bundle $\cl$ on a smooth curve $X$ is said to
 be {\it  $q$-very ample}, $q\ge 0$,  if $h^0 (X,\cl ) - h^0 (X,\cl(-F))= \deg F$ for
 any effective divisor $F$ with $\deg F \le q+1$.  Specifically, $0$-very ampleness and $1$-very ampleness  mean globally generatedness and   very ampleness, respectively. If $q\ge 1$, then
$\cl$ is  very ample and  the embedded curve $\varphi_{\cl} (X) \subseteq \P H^0(X,\cl)^*$  has no
 $n$-secant $(n-2)$-planes for any number $n\le q+1$, equivalently,
 dim${\langle F \rangle} _{\cl}=\deg F -1$ for any effective divisor $F$ on $X$
 with $\deg F \le q+1$.  Now,  we define an invariant to measure the linear position property of  $\varphi_{\cl} (X)$ in $\P H^0 (X, \cl)^*$.

\begin{Definition}\label{defor}
The order of very ampleness of a line bundle $\cl$  is defined  by
 $$Ova(\cl) :=\max \{~ q \in \Z ^{\ge  0} ~|~ \cl \mbox{ is } q\mbox{-very ample
 }\}.$$
 \end{Definition}

In the following theorem, we examine basic  properties of    presentations  of nonspecial line bundles.

\begin{Proposition}\label{thm0}
Let $\cl \simeq \ck _X -g^0_d +E$ be a presentation of a nonspecial line bundle  $\cl$ on a smooth curve $X$  of genus $g\ \ge 2$. Then we have the following.\\
\noindent{$(i)$}  $\cl\simeq \ck _X -g^0_t+F$ for a $F\ge 0$ with
$\deg F< \deg E$  if $g^0_d \neq 0$, $h^0 (X, \co _X  (E))\ge 2$.\\
\noindent{$(ii)$} $d \ge \deg E -1$ if  $\deg \cl \le 2g-1$. \\
\noindent{Specifically,} if $\cl \simeq \ck _X -g^0_d +E$ is a
minimal presentation then \\
\noindent{$(iii)$}   $\deg F \ge \deg E$ in case $h^0 (X, \cl )-h^0 (X, \cl (-F ))\le \deg
    F-1$,\\
\noindent{$(iv)$}  $Ova(\cl )=\deg E-2$,\\
\noindent{$(v)$}  $h^0 (X, \co _X  (E))=1$ in case  $d>0$,\\
\noindent{$(vi)$}   $d=0$  in case  $\deg \cl \ge 3g-2$, \\
 \noindent{$(vii)$} $d>0$ and $h^0 (X, \co _X  (E))=1$ in case $\deg E \ge 2$ and $\deg \cl \le 2g-1$, \\
\noindent{$(viii)$} $d\le g-1$;\ and  $d\le g-2$ in case $\deg E \ge 3$.\\
   \end{Proposition}

\begin{proof}
$($i$)$   Assume that $h^0 (X, \co _X  (E)) \ge 2$. For any $P \in \mbox{supp}(g^0_d)$
there is  an effective divisor $E' \simeq E$ with $(g^0_d, E' )\ge P$. Set $g^0_t :=
 g^0_d - (g^0_d, E' )$ and $F := E' - (g^0_d, E' )$.
Then there is another presentation  $\cl\simeq \ck _X -g^0_t+F$ with
$\deg F< \deg E$.

\noindent{$($ii$)$} For $\deg \cl \le 2g-1$,  the equality $\deg \cl =2g-2-d+\deg E$ gives $d\ge \deg E-1$.

\noindent{$($iii$)$} Assume that $h^0 (X, \cl )-h^0 (X, \cl (-F ))\le \deg
    F-1$ for a divisor $F$ on $X$. Then the Riemann-Roch
    Theorem gives the inequality $h^0 (X, \ck _X \otimes \cl
    ^{-1} (F))\ge 1$, which implies $\deg F \ge \deg E$ by the
    minimality of  $\cl \simeq \ck _X -g^0_d +E$.

\noindent{$($iv$)$} This result  follows from $($iii$)$ and the equality  $h^0 (X, \cl )-h^0 (X, \cl (-E ))= \deg    E-1$.

\noindent{$($v$)$} This is trivial by (i)

\noindent{$($vi$)$}  The condition $\deg \cl \ge 3g-2$ yields $\deg \cl \otimes \ck _X^{-1} \ge g$ and so there is an effective divisor $E$ such that $\cl \otimes \ck _X^{-1} \simeq \co _{X} (E)$, equivalently, $\cl \simeq \ck _X +E$, which is in  itself a minimal presentation.

\noindent{$($vii$)$}  Using (ii) and (v), we get $d>0$ and $h^0 (X, \co _X  (E))=1$ in case $\deg E \ge 2$ and $\deg \cl \le 2g-1$.

\noindent{$($viii$)$} Set $r:= h^0 (X, \cl ) -1.$  Choose a  $F \le G \in | \cl |$ with $\deg F= r+1$. Then
the divisor  $F$ satisfies  $ h^0 (X, \cl )-h^0 (X, \cl (-F )) \le
\deg F-1,$  which yields $\deg E\le \deg F=r+1$ by
$($iii$)$. Hence, we obtain $ d\le g-1$  since $r=(2g-2-d+\deg E)
-g$.

Assume $\deg E \ge 3$. Then $\cl$ is very ample. Since $r=\deg \cl -g$,  the condition $g\ge 2$ gives $\deg \varphi _{\cl} (X) \ge r+2$, whence the smooth curve $\varphi _{\cl} (X)$ has a $r$-secant $(r-2)$-plane by Lemma in \cite{KKM}. By (iii),  we have
$$\deg E\le r= (2g-2-d+\deg E  )-g,$$
which implies $d\le g-2$. Thus the result (viii) is verified.
 \end{proof}

\begin{Remark} (i) $\cl \simeq \ck _X -g^0_d +E$ is a minimal presentation if and only if $Ova (\cl) = \deg E -2$.\\
\noindent{$(ii)$} The minimal presentations of nonspecial line bundles $\cl$ with $\deg \cl \ge 3g-2$ are always trivial.  On the other hand, all the minimal presentations of globally generated nonspecial line bundles $\cl$ with $\deg \cl \le 2g-1$ are nontrivial since $\deg E\ge 2$ by being globally generated.\\
\noindent{$(iii)$} To arbitrary pair of effective divisors $g^0_d $ and $\xi^0_e$  with $ (g^0_d ,\xi^0_e ) =0$, we can associate a nonspecial line bundle $\cl \simeq \ck _X -g^0_d
+\xi^0_e$.  By Proposition \ref{thm0}, (vii),  it is enough to consider the divisors $g^0_d $ only in the range $d\le g-1$ (resp. $d\le g-2$)   for such construction of  (resp. very ample) nonspecial line bundles.\\
\noindent{$(iv)$}  If $\cl$ is minimally presented by $\ck _X -g^0_d+ E$ with
$\deg E\ge 3$, then the embedded curve
 $\varphi _{\cl} (X)$ has no $n$-secant $(n-2)$-planes for $n\le \deg E-1$. Moreover, if $\cl \simeq
\ck _X -g^0_d +E$ is
 a unique minimal presentation, then ${\langle E \rangle}_{\cl}$ is a unique $\deg E$-secant $(\deg E-2)$-plane.
% Thus, the nonspecial line bundles of given degree can be
% distinguished by $\deg E$ in minimal
% presentations.
\label{rmknon}
\end{Remark}

The following theorem plays a basic role in dealing with presentations of nonspecial line bundles.

\begin{Theorem}
 Assume that a nonspecial line bundle $\cl$ on a smooth curve $X$ has two different presentations
  $\cl \simeq \ck _X -g^0_d +E$ and $\cl \simeq \ck _X -g^0_t +F$.
  Then  $g^0_d +F \simeq  g^0_t +E$ but
    $g^0_d +F \neq  g^0_t +E$ as divisors.
    In particular, we have $h^0 (X, \co _X  (g^0_d +F ))\ge 2$.
  \label{lem1}
\end{Theorem}
\begin{proof}
 The equivalences  $\cl \simeq \ck _X -g^0_d +E \simeq \ck _X -g^0_t
 +F ~$
 imply that  $$g^0_d +F \simeq  g^0_t +E.$$  Assume that
 $g^0_d +F =  g^0_t +E$ as divisors. Then we get
 $$ g^0_t \le g^0_d ,~~ F \le E ,$$ according to  the condition
 $ ( g^0_t, F )=0$ for the presentation $\cl \simeq \ck _X -g^0_t +F$.
 The condition $ ( g^0_d, E )=0$ also gives  $$ g^0_d \le g^0_t ,~~E \le F . $$
  It is a contrary to the assumption that $\cl \simeq \ck _X -g^0_d +E$ and $\cl \simeq \ck _X -g^0_t +F$ are different.
  Hence, we have   $g^0_d +F \neq  g^0_t +E$ as divisors which implies 
    $h^0 (X, \co _X  (g^0_d +F ))\ge 2.$
 Thus the theorem is proved.
\end{proof}

\begin{Corollary}[Lemma 6, \cite{K1}]
 Let $\cl$  be a nonspecial line bundle  on a smooth curve $X$  which is presented by
 $\ck _X -g^0_d +E$  with  $\deg E\ge 3$. If  $\cl$ is not very ample,
 then there are   $g^0 _t\ge  0$ and  $P,~Q \in X$ such that   $g^0 _t+E \simeq g^0_d +P+Q$ and $g^0 _t +E \neq g^0_d +P+Q$ as divisors
 which implies $h^0 (X, \co _X  (g^0_d +P+Q))\ge 2$. \label{coro1}
\end{Corollary}

As we have seen,  admitting a  presentation $\cl \simeq \ck _X -g^0_d +E$ with $\deg E \ge 3$ does not guarantee the  very ampleness  of $\cl $. Thus we investigate  sufficient conditions for the very ampleness of  $\cl \simeq \ck _X -g^0_d +E$ with $d>0$ and $\deg E \ge 3$ in the following. Here, we  consider only  the case of $h^0 (X, \co _X (E))=1$ due to Proposition \ref{thm0}, (i).
\begin{Theorem}\label {va}
Let $X$ be a smooth curve of genus $g\ge 4$. And  let   $g^0_d$, $\xi ^0 _e$ be general  effective divisors on $X$ with $e\ge 3$ and $ ( g^0_d, \xi ^0 _e )=0$. \\
 \noindent{$(i)$} If $X$ is nonhyperelliptic, then $\cl \simeq \ck _X -g^0_d +\xi ^0 _e$  with $d\le g-3$ is very ample.\\
\noindent{$(ii)$} If  $X$ is hyperelliptic, then  $\cl \simeq  \ck _X -g^0_d +\xi ^0 _e$   with $d\le g-2$ is very ample.
\end{Theorem}

\begin{proof} Assume that $\cl \simeq \ck _X -g^0_d +\xi ^0 _e$ is not very ample.  Corollary \ref{coro1} gives   $$\co _X ( g^0 _t+\xi ^0_e ) \simeq \co _X  (g^0_d +P_1+P_2) \in  W^{\alpha}_{d+2}(X),~~ \alpha \ge 1 ,$$ for some   $P_1 ,~P_2 \in X$  and $g^0 _t$ on $X$, where  $$W^{\alpha}_{d+2}(X) :=\{\cl \in J (X)~|~h^0(X,\cl)\ge \alpha+1,~\deg \cl={d+2}\}.$$

\noindent{(i)} Let $X$ be a nonhyperelliptic curve. Due to the general choice of $g^0_d$ with $d\le g-3$ and H. Martens' Theorem(:(5.1) Theorem in \cite{ACGH}), we obtain
$$d-\alpha \le\dim W^{\alpha}_{d+2}(X) \le (d+2)-2\alpha -1, $$
whence $\dim W^{\alpha}_{d+2}(X) = (d+2)-2\alpha -1$ and $\dim |g^0_d +P_1+P_2|
= \alpha =1$.  According to Mumford's Theorem(:(5.2) Theorem in \cite{ACGH})  we have  one of   the following cases with a base locus $B$ :\\
\noindent{(Case 1)} $\phi : X \stackrel{3:1}\longrightarrow \P^1$  and $  | \co _X  (g^0_d +P_1+P_2) | =g^1_3 +B$.\\
\noindent{(Case 2)} $\phi : X \stackrel{2:1}\longrightarrow   \Gamma$ for an elliptic curve $\Gamma$ and  $ | \co _X  (g^0_d +P_1+P_2) | =\phi ^* g^1_2 +B$.\\
\noindent{(Case 3)} $X$ is a smooth plane quintic and $ | \co _X  (g^0_d +P_1+P_2) | =g^1_4 +B $.

First, consider  (Case 3). Note that every divisor of $g^1_4$ on $X$ is cut out by a line  in $P ^2$. By the general choices of $g^0_d$ and $\xi ^0_e$ we obtain
$$B= g^0 _{d-2} \mbox{ and } B \ge\xi ^0 _{e-2} $$ for some  $g^0_{d-2}<g^0_d$ and $\xi ^0_{e-2} < \xi ^0 _e$,  since
$ | \co _X  (g^0_d +P_1+P_2) | =| \co _X  (\xi ^0 _e  +g^0_t )| =g^1_4 +B$.  This implies  that $\xi ^0_{e-2}\le (g^0_d, \xi ^0 _e )=0$, which is contrary to $e\ge 3$.

 Also (Case 2) cannot happen by the  following. The general choices of $g^0_d$ and $\xi ^0 _e$ imply that
$\deg (g^0_d , \phi ^* (Q)) \le 1$ and $\deg (\xi^0_e , \phi ^* (Q)) \le 1 $
for any $Q \in \Gamma$.
 Thus we obtain $$B= g^0 _{d-2}, \  \ B \ge\xi ^0 _{e-2}$$ for some $g^0_{d-2}<g^0_d$ and $\xi ^0_{e-2} < \xi ^0 _e$, since $| \co _X  (g^0_d +P_1+P_2) |=|g^0 _t+\xi ^0_e | =\phi ^* g^1_2 +B$. This cannot occur for $(g^0_d, \xi ^0 _{e})=0$ with $e\ge 3$ as in (Case 3).

Finally, we are led to  (Case 1).  Since  $ | \co _X  (g^0_d +P_1+P_2) | =| \co _X  (\xi ^0 _e  +g^0_t )| =g^1_3 +B$,   the general choices of $g^0_d$ and $ \xi ^0 _e$ give $$B= g^0_{d-1}  ~\mbox{ and }~B \ge \xi ^0 _{e-1}   $$  for some $g^0_{d-1}\le  g^0_d$ and $\xi ^0 _{e-1} \le \xi ^0 _e$. This is a contradiction   to $(g^0_d ,  \xi ^0 _e ) =0$ with $e\ge 3$. As a consequence, the result (i) is valid.

\noindent{(ii)} Let $X$ be hyperelliptic. Due to the condition $d\le g-2$,  the linear system $ |\co _X  (g^0_d +P_1+P_2 ) | =g^{\alpha}_{d+2}$ with $\alpha \ge 1$ is special, and hence $$|\co _X  (g^0_d +P_1+P_2) | =| \co _X  (\xi ^0 _e +g^0_t )|=\alpha g^1_2 +B.$$
 Since  $ g^0_d$ is generally chosen,  only  two cases $\alpha =1$ and $\alpha =2$ can occur. According to the general choices of $ g^0_d$ and $\xi ^0 _e$, if
$\alpha =1$ then
$$B =g^0_{d- 1} +P_i ~\mbox{ and }~~\xi ^0 _{e-1} \le B \ \mbox{ for  some }\   g^0_{d- 1}\le g^0_d , \  \xi ^0 _{e-1} \le \xi ^0 _e ;$$
  if $\alpha =2$ then
$$ B =g^0_{d-2}~ \mbox{ and }~\xi ^0 _{e-2} \le B \ \mbox{ for  some }\ g^0_{d- 2}\le g^0_d, \ \xi ^0 _{e-2} \le \xi ^0 _e .$$  These are impossible for  $(g^0_d , \xi ^0 _e ) =0$ with $e\ge 3$. Consequently, the line bundle $\ck _X -g^0_d +\xi ^0 _e$ is very ample and hence the theorem is proved.
\end{proof}

\begin{Remark} Let $X$  be a smooth curve  of genus $g \ge 11$ and $\cl$ be a line bundle presented  by $\ck _X -g^0_d +\xi ^0 _e$ for general  $g^0 _{d}$ and $\xi ^0 _e$ with $d\le g-7$  and $e\ge 4$. Using Keem's Theorem in  \cite{Keem} which generalizes H. Martens' Theorem, we can similarly verify that  the embedded curve  $\varphi _{\cl}(X)$ has no 4-secant plane unless $X$ is either hyperelliptic,
trigonal, elliptic-hyperelliptic, a 4-sheeted covering of $\P^1$, or a 2-sheeted covering of a curve of genus 2.
\end{Remark}

\begin{Theorem}\label {corgon}
Let $\cl$  be a nonspecial line bundle  on a smooth curve $X$ which admits a
nontrivial presentation $\cl \simeq \ck _X -g^0_d +\xi^0_e$.

\noindent{$(i)$}  If $d+e\le gon(X)$, then $\cl \simeq \ck _X -g^0_d
+\xi^0_e$ is a minimal presentation.\\
\noindent{$(ii)$}  If $d+e< gon(X)$, then $\cl \simeq \ck _X -g^0_d
+\xi^0_e$ is a unique minimal presentation.\\
\noindent{$(iii)$}  If  $\cl \simeq \ck _X -g^0_d +\xi^0_e$ with $d+e=\gon (X)$   admits other presentations $\cl \simeq \ck _X -h^0_{d,j} +\zeta ^o _{e,j}$ of type $(d,e)$ for $j\in J$, then all the  pencils $| \co _X (g^0_d +\zeta ^o _{e,j}) |\  ( = | \co _X (h^0_{d,j} +\xi ^0_e)|)$ with $j \in J$ are mutually distinct $g^1_{\gon (X)}$ on $X$.
%\noindent{$(iv)$}  If $X$ is an $n$-gonal curve with a unique $g^1_n$ and  $g^0_{n-e} +\zeta ^0_e,   ~ h^0_{n-e} +\xi^0_e$ are distinct divisors in $ g^1_n$ with $ (h^0_{n-e} , \zeta ^0_e )=0$, then $\cl \simeq \ck _X -h^0_{n-e}+\zeta ^o _e$ is only the different minimal presentations from $\cl \simeq \ck _X -g^0_{n-e} +\xi ^0 _e$.
 \end{Theorem}
\begin{proof}  (i) If $\cl \simeq \ck _X -g^0_d +\xi^0_e$ is not a
minimal presentation, then there is a  $\xi^0_s$ with $ s<e$ such
that $h^0 (X, \co _X  ( g^0_d +\xi^0_s ))\ge 2$  by Theorem
\ref{lem1}. It cannot occur for $d+e\le gon(X)$.

\noindent{(ii)} In case  $d+e<gon(X)$,   Theorem
\ref{lem1} also implies that there is no another
presentation $\cl \simeq \ck _X -h^0_d +\zeta ^0_e$.

\noindent{(iii)} Assume $\cl \simeq \ck _X -g^0_d +\xi^0_e$ with $d+e =\gon (X)$ admits two other different  presentations $\cl \simeq \ck _X -h^0_d +\zeta ^o _e\simeq \ck _X -f^0_d +\tau ^o _e$.   Due to Theorem \ref{lem1},  we have the following:
$$g^0_d +\zeta ^o _e \simeq   h^0_d +\xi ^o _e, ~~ g^0_d +\zeta ^o _e \neq   h^0_d +\xi ^o _e,$$
$$g^0_d +\tau ^o _e \simeq  f^0_d+\xi ^o _e, ~~ g^0_d +\tau ^o _e \neq  f^0_d+\xi ^o _e,$$
and  both   $| \co _X (g^0_d +\zeta ^o _e) |$ and  $| \co _X (g^0_d +\tau ^o _e) |$
are pencils of degree $\gon (X)$.
If   $| \co _X (g^0_d +\zeta ^o _e) |=| \co _X (g^0_d +\tau ^o _e) |,$
then
$$g^0_d +\zeta ^o _e \simeq g^0_d +\tau ^o _e \  \mbox{ and } \   h^0_d +\xi ^o _e \simeq  f^0_d+\xi ^o _e ,$$
whence $\zeta ^o _e =\tau ^o _e$ and  $h^0_d = f^0_d$ since both of $d$ and $e$ are smaller than  $\gon (X)$ by the conditions  that $d+e =\gon (X)$,
$g^0_d > 0$ and  $\xi ^o _e > 0$.  
It is a contradiction to the assumption that $\cl \simeq \ck _X -h^0_d +\zeta ^o _e$ and $\cl \simeq \ck _X -f^0_d +\tau ^o _e$ are distinct. Hence   two pencils $| \co _X (g^0_d +\zeta ^o _e) |$ and  $| \co _X (g^0_d +\tau ^o _e) |$ are different. This gives the result (iii).
 Thus we complete the proof of the theorem.
\end{proof}
\begin{Corollary} Let $X$ be  an $n$-gonal curve. For  $0<e<n$,  choose two distinct divisors $g^0_{n-e} +\zeta ^0_e,   ~ h^0_{n-e} +\xi^0_e \in g^1_n$ with  $(g^0_{n-e}, \xi^0_e )=0$  and $ (h^0_{n-e} , \zeta ^0_e )=0$. Then  we have a nonspecial line bundle $\cl$ with $Ova (\cl) = e-2$ which is distinctly presented by $ \ck _X -h^0_{n-e}+\zeta ^o _e$ and $ \ck _X -g^0_{n-e} +\xi ^0 _e$. Moreover, if $X$ has a unique $g^1_n$, then $\cl \simeq \ck _X -h^0_{n-e}+\zeta ^o _e$ is  the only different minimal presentations from $\cl \simeq \ck _X -g^0_{n-e} +\xi ^0 _e$.
\label{gonpre}
\end{Corollary}
%Due to Corollary \ref{gonpre}, an $n$-gonal curve $X$ carries  nonspecial line bundle $\cl$ with $Ova (\cl) = e-2$ for any $0<e<n$, since general two distinct divisors $g^0_{n-e} +\zeta ^0_e,   ~ h^0_{n-e} +\xi^0_e \in g^1_n$ satisfy  the conditions $(g^0_{n-e}, \xi^0_e )=0$  and $ (h^0_{n-e} , \zeta ^0_e )=0$.

\begin{Remark} (i) Whenever we take  arbitrary  $g^0_d$ and $\xi ^o _e$  on  $X$  with $d+e \le gon(X)$ and $(g^0_d , \xi ^o _e )=0$, we obtain a line bundle $\cl \simeq \ck_X -g^0_d +\xi ^o _e$ which  is in  itself  a  minimal presentation, equivalently,   $Ova (\cl ) =e -2$. In particular, if $e \ge 3$(:$\cl$ is very ample) and $d+e < \gon (X)$, then ${\langle \xi^0_e \rangle}_{\cl}$  is a unique $e$-secant $(e- 2)$-plane and has no $s$-secant $(s-2)$-planes for any $s \le e -1$.\\
%\noindent{(ii)}  Let $X$ be an $n$-gonal curve and let $e$ be a number with  $0<e<n$. Then, $X$ carries a nonspecial line bundle $\cl$ with   $\deg \cl =2g-2-n+2e$ which admits two different minimal presentations with type $(d,e)$, since general two distinct divisors $g^0_{n-e} +\zeta ^0_e,   ~ h^0_{n-e} +\xi^0_e \in g^1_n$ satisfy  the conditions $(g^0_{n-e}, \xi^0_e )=0$  and $ (h^0_{n-e} , \zeta ^0_e )=0$.\\
\noindent{(ii)}  Let $X$ be an $n$-gonal curve with a unique $g^1_n$. For any number $e$ with  $0<e<n$,
  $X$  has infinitely many nonspecial line bundles $\cl$ satisfying   $\deg \cl=2g-2-n+2e$  and $ Ova(\cl )=e-2$  by Corollary \ref{gonpre}, since two different  general divisors $g^0_{n-e} +\zeta ^0_e,   ~ h^0_{n-e} +\xi^0_e \in g^1_n$ satisfy  the conditions $(g^0_{n-e}, \xi^0_e )=0$  and $ (h^0_{n-e} , \zeta ^0_e )=0$.   In the case $e\ge 3$,  the embedded curve $\varphi _{\cl } (X)$ has exactly two $e$-secant $(e-2)$-planes and has no $s$-secant $(s-2)$-planes for $s\le e-1$.\\
\noindent{(iii)} If $\cl \simeq \ck_X -g^0_d +\xi ^o _e$ with $d+e = \gon (X)$,  then Theorem \ref{corgon} (iii) implies  the following inequality:
 \begin{eqnarray*}
& &\# \{   \ e\mbox{-secant }(e-2)\mbox{-planes of } \varphi _{\cl} (X)\   \} \\
& &\le 1 +  \#  \{ \ g^1_{\gon (X)}  \mid h^0 (X, g^1_{gon(X)}(-g^0_d))\ge 1, \ h^0 (X, g^1_{gon(X)}(-\xi^0_e))\ge 1\  \}.
\end{eqnarray*}
\label{rmkcan}
\end{Remark}
We can see  the exactness of  the condition $d+e \le \gon(X)$ in Theorem \ref{corgon} through the following example.
\begin{Example} Let $X$ be a smooth curve.
Choose   two distinct general divisors $g^0_d +\xi ^0 _{e-b} , ~g^0_{d-b} +\xi ^0 _{e} \in g^1 _{\gon (X)}$ with $b>0$. Note that  $d+e =\gon (X) +b>\gon (X)$ and an equivalence $\ck _{X} - g^0_{d-b} +\xi ^0 _{e-b} \simeq \ck _{X} - g^0_{d} +\xi ^0 _{e} $. 
The general choices imply  that $(g^0_d , \xi ^0 _{e} )=0$ and   $(g^0_{d-b} , \xi ^0 _{e-b} )=0$ and hence $\cl \simeq \ck _{X} - g^0_{d} +\xi ^0 _{e} \simeq 
\ck _{X} - g^0_{d-b} +\xi ^0 _{e-b} $ are well defined presentations.  This means that $\cl \simeq \ck _{X} - g^0_{d}  +\xi ^0 _{e} $ is not minimal. 
%Note that  $d+e =\gon (X) +b>\gon (X)$.  The general choices imply  that $(g^0_d , \xi ^0 _{e} )=0$ and   $(g^0_{d-b} , \xi ^0 _{e-b} )=0$. And we have an equivalence $\ck _{X} - g^0_{d-b} +\xi ^0 _{e-b} \simeq \ck _{X} - g^0_{d} +\xi ^0 _{e} $, whence the presentation $\cl \simeq \ck _{X} - g^0_{d}  +\xi ^0 _{e} $ is not minimal.
\label{exgon}
\end{Example}
The following example is comparable to  Remark  \ref{rmkcan}, $(iii)$.
\begin{Example} Let $X$ be an $m$-fold covering of an elliptic curve $\Gamma$ via $\phi :X \to \Gamma$. Assume that $\cl \simeq \ck _X -\phi^ * (P)+\phi^ * (Q)$ for  two  distinct points $P$, $Q\in  \Gamma$.  Then, for an arbitrary  $R \in \Gamma$, there is a point $S\in \Gamma$ such that  $\cl\simeq  \ck _X -\phi^ * (S)+\phi^ * (R)$. Specifically, if $\gon(X) =2m$, then there are  infinitely many $g^1_{gon(X)}$ on $X$ and the line bundle $\cl$ has  infinitely many minimal presentations of  type $(m,m)$ with $2m=\gon(X)$.
\label{egel}
\end{Example}
\begin{proof} Let $R$  be an arbitrary  point of  $ \Gamma$ with $R\neq Q$.  Since $h^0 (\Gamma, \co _{\Gamma} (P+R)) = 2$, we can choose a point $S$ of  $ \Gamma$ such that $P+R \simeq Q+S$. Then we have
 $$\phi^ * (P) +\phi^ * (R ) \simeq \phi^ * (Q)+\phi^ * (S ),$$
which gives   $$\cl \simeq\ck _X -\phi^ * (P) +\phi^ * (Q) \simeq \ck _X-\phi^ * (S ) +\phi^ * (R ).$$ Thus the result follows.
\end{proof}

In addition, 
%we examine the difference between nonspecial line bundles and 
observe the line bundles of degree $2g-2$ with respect to secant properties, since it is interesting to distinguish properties of  a nonspecial line bundle of that degree from the special one which is canonical.

\begin{Remark} Let $\cl$ be a very ample line bundle with $\deg \cl =2g-2$.\\
\noindent{(i)} 
%The special case:  $X$ is  a non-hyperelliptic curve and    $\cl$ is the canonical  line bundle
 The special case:   $\cl$ is the canonical  line bundle $\ck _X$ on a nonhyperelliptic curve $X$, for which the embedded curve $\varphi _{\ck _X} (X)$ has  at least one-dimensional family of $\gon (X)$-secant $(\gon (X)-2)$-planes but has no $s$-secant $(s-2)$-planes for $s \le \gon (X) -1$.

\noindent{(ii)} The nonspecial case:  Let $e$ an arbitrary number  with $3\le e\le \frac{\gon (X) }{2}$  $($resp. $ 3\le e< \frac{\gon (X) }{2})$. By Corollary \ref{gonpre}, $X$ carries very ample nonspecial line bundles $$\cl \simeq \ck_X -g^0_e +\xi ^o _e ,$$  for each of which   $\varphi _{\cl} (X)$ has an (resp. unique )  $e$-secant $(e-2)$-plane  but has no  $s$-secant $(s-2)$-planes for any $s \le e-1$.  Furthermore, if $X$ is a general $k$-gonal, then   the  range of the number $e$ can be  extended  up to $3\le e<\frac{g+1}{4}$(see  Corollary \ref{coro4} and Remark \ref{genrmk}).

\noindent{(iii)} Assume that $X$ is an $m$-fold covering of an elliptic curve $\Gamma$ with $\gon(X) =2m$. Let $\cl \simeq \ck _X -\phi^ * (P)+\phi^ * (Q)$  for  two distinct points $P$, $Q\in  \Gamma$.  Example \ref{egel} implies that the curve $\varphi _{\cl} (X)$ has no $n$-secant $(n-2)$-planes for $n\le m-1$ and has infinitely many $m$-secant $(m-2)$-planes such that
$$\{  \  m\mbox{-secant } (m-2)\mbox{-planes}   \  \} =\{ {  \ \langle \phi^ * (R) \rangle}_{\cl} ~|~ R \in \Gamma    \ \}.$$ Consequently, the curve $\varphi _{\cl} (X)$ lies on an $(m-1)$ dimensional scroll  $S$ over  $\Gamma$.  Moreover,  any two distinct $m$-secant $(m-2)$-planes have no common points,
% since the condition  $\gon(X)=2m$ implies $h^0 (X, \co _X (\phi ^ * (P) + \phi^ * (R))=h^0 (\Gamma, \co _{\Gamma} (P +R ))=2$ and thus  $\dim {\langle \phi^ * (R_1 +R_2) \rangle}_{\cl} =h^0 (X, \cl ) -h^0 (X, \cl (-R_1 -R_2)) -1 =2m-3 $ by  the Riemann-Roch Theorem and $\cl \simeq \ck _X  ????$.
 since   the Riemann-Roch Theorem gives $\dim {\langle \phi^ * (R_1 +R_2) \rangle}_{\cl}  =2m-3 $ due to $\cl (-R_1 -R_2) \simeq \ck _X  -\phi^ * (P) - \phi^ * (R_2)$ and $h^0 (X, \co _X (\phi ^ * (P) + \phi^ * (R_2)))=2$.
% the condition  $\gon(X)=2m$ implies $h^0 (X, \co _X (\phi ^ * (P) + \phi^ * (R))=h^0 (\Gamma, \co _{\Gamma} (P +R ))=2$ and thus  $\dim {\langle \phi^ * (R_1 +R_2) \rangle}_{\cl} =h^0 (X, \cl ) -h^0 (X, \cl (-R_1 -R_2)) -1 =2m-3 $ by  the Riemann-Roch Theorem and $\cl \simeq \ck _X  ????$.
\noindent{(iv)} Assume that $X$ is a  double covering of a smooth curve $Y$ of genus $\gamma$ via $\phi : X \to Y$.  For each $e \le \frac{g- 2\gamma}{2} $, there are  nonspecial line bundles  
  $$\cl \simeq \ck _X -g^0_{e} + \xi^0_e ,$$   
for which   $\varphi _{\cl} (X)$ has an   $e$-secant $(e-2)$-plane  but has no  $s$-secant $(s-2)$-planes for any $s \le e-1$(see Corollary \ref{coro5}). 
\end{Remark}

Note that one of the natural ways to construct nonspecial line bundles on a smooth curve $X$  is  to use a $g^r_d$,  whose existence on $X$  is already well known, such as $\cl \simeq \ck _X -g^r_d +\xi ^0 _e$  for some  $\xi ^0 _e$  on $X$.  In this case, 
%if  $\beta := \mbox{max} \{ \deg (\xi ^0 _e , D) | D \in g^r_d \}$ then
 we may  raise a question concerning the minimality of $\cl \simeq \ck _X -g^0_{d-\beta} +\xi ^0 _{e-\beta} $, where 
%$\deg  (\xi ^0 _e , D)=\beta$,
$g^0_{d-\beta} := g^r_d - (\xi ^0 _e , {\tilde D})$ and $\xi ^0 _{e-\beta}:=\xi ^0 _e -(\xi ^0 _e , {\tilde D})$ for some ${\tilde D}\in g^r_d$  with $\deg  (\xi ^0 _e , {\tilde D}) = \beta$, $\beta =:\mbox{max} \{ \deg (\xi ^0 _e , D) | D \in g^r_d \}$.
The following theorem gives sufficient conditions for the minimality.

\begin{Theorem} \label{submin}
Let $X$ be a smooth curve with a complete $g^r_d$ and  let $\cl$ be a nonspecial line bundle on $X$  given by  $\ck _X -g^r_d +\xi ^0 _e$ for some $\xi ^0 _e$ on $X$. Set $\beta := \mbox{max} \{ \deg (\xi ^0 _e , D) | D \in g^r_d \}$.  
Assume that $\dim |g^r_d +G| =r$ for any  $G\ge 0$ with $\deg G \le  e-\beta -1$.  If  $D \in g^r_d$ satisfies  $\deg (\xi ^0 _e , D)  =\beta$ then   $ \cl \simeq \ck _X -g^0_{d-\beta} + \xi ^0 _{e-\beta}$ is a minimal presentation, where $g^0_{d-\beta}:=g^r_d -(\xi ^0 _e , D),~   \xi ^0 _{e-\beta}:= \xi ^0 _e - (\xi ^0 _e , D)$.
\end{Theorem}

\begin{proof}
 Since $\cl$ is nonspecial,  we obtain $\beta \le e-1$, and  $ \cl \simeq \ck _X -g^0_{d-\beta} + \xi ^0 _{e-\beta}$ is  a minimal presentation in case $\beta = e-1$. 
%Nonspeciality of  $\cl$ gives $\beta \le e-1$.  
Assume that  for $\beta \le e-2$ there is another presentation $\cl \simeq \ck _X - h^0 _t +\zeta ^0 _s$ with $s\le e-\beta -1$. This gives  
$\ck _X -g^r_d +\xi ^0 _e \simeq \ck _X - h^0 _t +\zeta ^0 _s$, whence 
 $$|g^r_d +\zeta ^0 _s| = | h^0 _t  + \xi ^0 _e|.$$ 
%from  the equivalences
%  there are  $h^0 _t$ and $\zeta ^0 _s$  on $X$ with $s\le e-\beta -1$ and $(h^0 _t, \zeta ^0 _s)=0$ such that 
%$\cl \simeq\ck _X -g^r_d +\xi ^0 _e \simeq \ck _X - h^0 _t +\zeta ^0 _s.$
% whence $|g^r_d +\zeta ^0 _s| = | h^0 _t  + \xi ^0 _e|$.  
Since  $\dim |g^r_d +G| =r$ for any  $G\ge 0$ with $\deg G \le e-\beta -1$, we obtain $$| h^0 _t  + \xi ^0 _e| = |g^r_d +\zeta ^0 _s| =g^r_d +\zeta ^0 _s ,$$
which means  that $\zeta ^0 _s \le h^0 _t  + \xi ^0 _e$ and thus  $\zeta ^0 _s \le \xi ^0 _e$ due to $(h^0 _t, \zeta ^0 _s)=0$. From the equality $| h^0 _t  + \xi ^0 _e| =g^r_d +\zeta ^0 _s $ we get  $\xi ^0 _e -\zeta ^0 _s \le F \in g^r_d$ and so  $(\xi ^0 _e -\zeta ^0 _s , F) \le (\xi ^0 _e, F )$, which is contrary to the definition of $\beta$ since $ e-s \ge \beta +1$ for $s\le e-\beta -1$. Thus $ \cl \simeq \ck _X -g^0_{d-\beta} + \xi ^0 _{e-\beta}$ is a minimal presentation.
\end{proof}

On the other hand, there is an example such that the minimality fails when    $\dim |g^r_d +G| >r$ for some  $G\ge 0$ with $\deg G \le e-\beta -1$.

\begin{Example} \label{subex}
Let $X$ be a linearly normal smooth curve of type $(a,b)$ with $a\ge b \ge 2$ on a smooth quadric surface in $\P ^3$.
Choose a subdivisor $\xi ^0 _e$ of a general $H\in | \co _X (1)|$ with $e\ge a +1$. 
%=:g^3_{a+b}
Let $\cl \simeq \ck _X -g^1_a +\xi ^0 _e$.  
%and $\beta := \mbox{max} \{ \deg (\xi ^0 _e , D) | D \in g^1_a \}$.
 Then,  $\mbox{max} \{ \deg (\xi ^0 _e , D) | D \in g^1_a \}$ is equal to one. However the  presentation $\cl \simeq \ck _X -g^0_{a -1} +\xi ^0 _{e-1}$ with $ \xi ^0 _{e-1} :=\xi ^0 _e -P \ge 0$ and $g^0_{a -1}  := g^1_a -P$  is not minimal.
\end{Example}
\begin{proof} The general choice of $H$ implies that $\mbox{max} \{ \deg (\xi ^0 _e , D) | D \in g^1_a \} =1$ and $h^0 (X, \co _X (H-\xi ^0 _e))=1$.  Take a $Q \in X$ with $Q \le ( H-\xi ^0 _e)$. Set $ h^0 _{a+b-e-1} := H-\xi ^0 _e  -Q, ~ \zeta ^0 _{b-1} := g^1_b -Q$. From the equivalence  $ (H-\xi ^0 _e  -Q) +\xi ^0 _e \simeq g^1_a  + ( g^1_b -Q)$  we obtain
  $$\ck _X -g^1_a +\xi ^0 _e \simeq \ck _X -h^0 _{a+b-e-1} +\zeta ^0 _{b-1}.$$ 
This means that $\cl \simeq \ck _X -g^0_{a -1} +\xi ^0 _{e-1}$ is not a  minimal
presentation for $b\le a \le e-1$.
\end{proof}

Note that this example  does not satisfy the hypothesis of Theorem \ref{submin}, since  $\dim |g^1_a +G | =\dim |H-Q|=2$ for $G:=g^1_b -Q$ of degree $b-1 \le e-1-\beta$.

\section{Minimal presentations  on  multiple coverings}

%Throughout  this section, $X$ is a smooth curve of genus $g\ge 2$ which
%admits an $n$-fold covering morphism $\phi :X \rightarrow Y$ for a
%smooth curve $Y$ of genus $\gamma$.
 In this section, we investigate  sufficient conditions for  the minimality of  presentations of  nonspecial line bundles  on  multiple coverings.

 Since every trivial presentation is minimal, we   consider only  nontrivial presentations(: $g^0_d \neq 0$) and so we use a notation $\xi ^0_e$ instead of  $E$ due to Proposition \ref{thm0}, (v). Thus the aim of this section is to explore  sufficient conditions for  the minimality of nontrivial presentations such as $\cl \simeq \ck _X -g^0_{d} +\xi^0_e$  on  multiple coverings.
We also assume $e\ge 2$ which is necessary for the line bundle $ \ck _X -g^0_{d} +\xi^0_e$ to be   globally generated.

First, we  examine necessary conditions for $\cl \simeq \ck _X -g^0_{d} +\xi^0_e$ to be minimal on curves $X$ with $\phi : X \to \P ^1$ which are the simplest coverings to deal with.

\begin{Proposition}
Let  $X$ admit an $n$-fold covering $\phi:X \to \P^1$ and let $ \cl$ be
a nonspecial   line bundle   on $X$.  If  $\cl\simeq \ck_X-g^0_d+\xi^0_e$ is a minimal presentation, then 
$\deg  (g^0_d , \phi ^* (P)) + \deg (\xi^0_e, \phi ^* (Q))\le n$ for  any $P$, $Q \in \P ^1.$
 \label{thmgon}\end{Proposition}
\begin{proof} Let $\cl\simeq \ck_X-g^0_d+\xi^0_e$ is a minimal presentation. Suppose that $(g^0_d , \phi ^* (P) ):=D>0$ and $(\xi^0_e, \phi ^* (Q)):=E >0$ for $P,Q \in \P ^1$.  
The equivalence $\phi ^* (P ) \simeq \phi ^* (Q )$ gives
$$\ck_X-g^0_d+\xi^0_e \simeq  \ck _X-g^0_d +D-(\phi ^* (Q ) -E) +\xi ^0_e-E+( \phi ^* (P )-D).$$ If we set
$E' : = \phi ^* (P ) -D, ~~ D' := \phi ^* (Q ) -E,$ then we have $$\cl \simeq  \ck _X-(g^0_d -D +D') +(\xi ^0 _e -E +E'),$$  with $g^0_d -D +D' \ge 0 \mbox{ and } \xi ^0 _e -E +E' \ge 0$. Take a divisor $B\ge 0$ such that 
$|g^0_d -D +D' -B | =:g^0_t$ and $ | \xi ^0 _e -E +E'  -B| :=\xi ^0_s $ with 
$(g^0_t, \xi ^0_s)=0$. This gives  a presentation 
$$\cl \simeq  \ck _X-g^0_t +\xi ^0_s ~ \mbox{ with } s+t \le d+e -\deg (D +E) +\deg (D' +E'),$$
whence the minimality of  $\cl\simeq \ck_X-g^0_d+\xi^0_e$ implies $\deg (D +E) \le \deg (D' +E') .$ This  yields $\deg (D +E) \le n$ since $\phi ^* (P +Q)=D+D'+E+E'$.
%
% Suppose that $\deg  (g^0_d , \phi ^* (P)) + \deg (\xi^0_e, \phi ^* (Q))\ge n+1$ for  some $P$, $Q \in \P ^1.$ Then 
%% Since  $\deg ( (g^0_d , \phi ^* (P)) + (\xi^0_e, \phi ^* (Q)))\ge n+1$ for  some $P$, $Q \in \P ^1$, 
% there are effective divisors $D\le g^0_d$, $E\le \xi ^0 _e$  such that  $D \le \phi ^* (P)$,  $E \le \phi ^* (Q)$ and $\deg (D+E) \ge n+1$.
% The equivalence $\phi ^* (P ) \simeq \phi ^* (Q )$ gives
%$$\ck_X-g^0_d+\xi^0_e \simeq  \ck _X-g^0_d +D-(\phi ^* (Q ) -E) +\xi ^0_e-E+( \phi ^* (P )-D).$$ If we set
%$E' : = \phi ^* (P ) -D, ~~ D' := \phi ^* (Q ) -E,$ then we have $$\cl \simeq  \ck _X-(g^0_d -D +D') +(\xi ^0 _e -E +E'),$$  with $g^0_d -D +D' \ge 0 \mbox{ and } \xi ^0 _e -E +E' \ge 0$. Then, we can take a divisor $B\ge 0$ such that 
%$|g^0_d -D +D' -B | =:g^0_t$ and $ | \xi ^0 _e -E +E'  -B| :=\xi ^0_s $ with 
%$(g^0_t, \xi ^0_s)=0$. Finally, we
%
%By subtracting an effective divisor from both $|g^0_d -D +D' |$  and    $ | \xi ^0 _e -E +E' |$,  we  obtain  a presentation $$\cl \simeq  \ck _X-g^0_t +\xi ^0_s ~
%%for some  $g^0 _t \le g^0_d -D +D'$ and $\xi ^0_s \le \xi ^0 _e -E +E'.$  Then we have
% \mbox{ with } s+t \le d+e -\deg (D +E) +\deg (D' +E') .$$ The condition $\deg (D +E )\ge n+1$ gives $$\deg (D' +E') =\deg( (\phi ^* (P +Q) - (D+E) ) <\deg (D+E), $$
%since $E' : = \phi ^* (P ) -D, ~ D' := \phi ^* (Q ) -E.$ Consequently, we obtain
%$$ s+t <d+e,$$ and hence the equalities  $\deg \cl =2g-2 -d +e=2g-2-t+s$ give    $$s <e.$$
% This is contrary to the minimality of $\cl\simeq \ck_X-g^0_d+\xi^0_e$. 
 Accordingly the theorem is verified.
\end{proof}

%\begin{Definition} A globally generated line bundle $\cl$ on a multiple covering $X$ via $\phi :X \rightarrow Y$  is said to be {\it composed with the covering morphism $\phi$} if
%its associated morphism $\varphi _{\cl}$ factors through $\phi$.
%\end{Definition}
In fact, the  conclusion that $\deg  (g^0_d , \phi ^* (P)) + \deg (\xi^0_e, \phi ^* (Q))\le n$ for any $P,Q$ of the base curve  also becomes a sufficient condition for the minimality of $\cl\simeq \ck_X-g^0_d+\xi^0_e$ on multiple coverings in some restricted  range of $d+e$ as follows.

\begin{Theorem}\label{mainthm}
Assume  that $X$ admits an  $n$-fold covering morphism $\phi :X \to Y$ for
smooth curve $Y.$    Choose  $g^0_{d}$ and $\xi^0_e$  on $X$ with $ ( g^0_d , \xi^0_e )=0$, $e\ge 2$  and  $d+e \le \mu ,$ where $ \mu := \mbox{min}\{ \deg \cn ~| ~\cn :  \mbox{globally generated and not composed with }\phi  \}$.
If $\deg (g^0_d , \phi ^* (P)) + \deg (\xi^0_e, \phi ^* (Q))\le n$ for  any  $P$, $Q \in Y$,  then we have a nonspecial line bundle $\cl\simeq \ck_X-g^0_d+\xi^0_e$ which  is   in itself a minimal presentation.
\end{Theorem}
\begin{proof}
Suppose that $\cl \simeq \ck _X -g^0_{d} + \xi^0_e$ has another
presentation $\cl \simeq \ck _X -h^0_{t} + \zeta^0_s$ with $s\le e-1$. The condition $s\le e-1$ also implies $t\le d-1$ since  $d-e =t-s$.
According to Theorem \ref{lem1}  we have
$$g^0_d +\zeta ^o _s \simeq   h^0_t +\xi ^o _e, ~~ g^0_d +\zeta ^o _s \neq   h^0_t +\xi ^o _e, ~~h^0 (X,\co _X (g^0_{d} + \zeta^0_s )) \ge 2, $$
whence
$$|g^0_{d} + \zeta^0_s  |= |g^0_{t} +
 \xi^0_e |= \phi^* ( g^r_m ) +B, ~ r\ge 1, ~B: \mbox{  base locus},$$
since $d+s < d+e \le \mu $.
Then we have the  following  decompositions:
\begin{align}\label{1-1}
&g^0_{d}=\Lambda  +B_1 ,~~~  h^0_{t}=\Lambda ' +B' _1,\\
&\zeta^0_s =\Sigma '  +B' _2 ,~~~\xi^0_e =\Sigma +B _2,
\nonumber
\end{align}
such that
$$ | \Lambda  +\Sigma '  |=| \Lambda ' +\Sigma |=\phi^* ( g^r_m ) , $$
$$B=B_1+B' _2=B' _1+B _2,$$
for some effective divisors
$\Lambda,~ \Lambda ' ,~ \Sigma , ~ \Sigma ' $, $B_k$, $B' _k$ $k =1,2$.
%$\Lambda  ,~ \Lambda ' ,~B_1,~B' _1,~ \Sigma , ~ \Sigma '
%,~\Sigma  _2, ~B' _2.$
Thus there are  points $P_i $, $Q_j \in Y$, $i, j=1,...,m$  such that $$\Lambda  + \Sigma '  = \phi ^* ( P_1 +...+P_m ) \mbox{ and  } \Lambda ' +\Sigma  = \phi ^* ( Q_1 +...+Q_m ).$$
Accordingly, for each $i, j\in\{ 1,...,m \}$ we can set
$$ D_i +E' _i = \phi ^* ( P_i)  ~ \mbox{ for some } D_i \le \Lambda   ,  ~ E'_i \le \Sigma '  ,$$
$$ D ' _j +E _j  = \phi ^* ( Q_j)    ~\mbox{ for some } D'_j \le \Lambda '  ,  ~ E_j \le \Sigma .$$
Thus  the hypothesis on $g^0_d+\xi^0_e$  in the theorem gives
 $$ \deg (D_i +E _j ) \le \deg (g^0_d , \phi ^* (P_i)) + \deg (\xi^0_e, \phi ^* (Q_j))\le n,$$ whence $$\deg (D'_j +E' _i )\ge n \ge\deg (D_i +E _j ) $$
due to  $ D_i +E_j + D'_j +E'_i = \phi ^* (P_i +Q_j ) $. This implies
\begin{equation}  \label {*}\deg (\Lambda ' + \Sigma '  ) \ge \deg (\Lambda  + \Sigma ) , \end{equation}
since  $\Lambda  + \Sigma =\sum ^n_{i=1} D_i +\sum ^n_{j=1}E _j $ and $\Lambda ' + \Sigma ' =\sum ^n_{j=1} D'_j +\sum ^n_{i=1} E' _i $.

On the other hand, because  $B$ is a base locus,  we have $B=B_1+B' _2=B' _1+B _2$ as divisors, whence
$$B' _1 = B_1, ~~~B '_2 =B _2 ~~~~ \mbox{ as divisors } $$ by the conditions $ ( g^0_d , \xi^0_e )=0$ and $ ( h^0_t , \zeta^0_s )=0$.  Accordingly, by (\ref{1-1}) we obtain
$$\deg \Lambda  =d- \deg B_1 \ge 1+  t -\deg B' _1 =1+ \deg \Lambda ' $$
$$\deg \Sigma =e- \deg B_2 \ge  1+ s -\deg B' _2 =1+ \deg \Sigma '  $$ for   $d\ge t+1$  and $e\ge s+1$.  This gives that  $$\deg (\Lambda  + \Sigma ) \ge 2+ \deg (\Lambda ' + \Sigma '  ),$$
% On the other hand, since  $\Lambda  + \Sigma =\sum ^n_{i=1} D_i +\sum ^n_{j=1}E _j $ and $\Lambda ' + \Sigma ' =\sum ^n_{j=1} D'_j +\sum ^n_{i=1} E' _i $, we have
%$$\deg (\Lambda  + \Sigma )=\deg (\Lambda ' + \Sigma '  )$$
% since $\deg (D'_j +E' _i ) \ge\deg (D_i +E _j ) $ for each $i, j=1,...,m$.
which is contrary to (\ref{*}).  Thus $\cl \simeq \ck _X -g^0_{d} + \xi^0_e$ is a minimal presentation.
\end{proof}
\begin{Remark}  Whenever we take  general $g^0_d$ and $\xi^0_e$ on a multiple covering $X$ with  $d+e \le \mu$ and  $(g^0_d,\xi^0_e )=0$,  we obtain a nonspecial line bundle $\cl \simeq \ck _X -g^0_{d} +\xi^0_e$ which
 is in  itself a minimal presentation since  the general choices of $g^0_d$ and $\xi^0_e$ imply $\deg (g^0_d , \phi ^* (P)) + \deg (\xi^0_e, \phi ^* (Q))\le n$ for  any $P$, $Q \in Y$.
\label{genrmk}
\end{Remark}

 \begin{Corollary} Let $X$ be a general $n$-gonal curve of genus $g\ge 4$ via
 $\phi :X \rightarrow \P^1 $.  Choose  $g^0_{d}$ and  $ \xi^0_e$ on X   with $ (g^0_d , \xi ^0_e) =0$ and $d+e \le \frac{g+3}{2}$. Then $\cl \simeq \ck _X-g^0_{d} + \xi^0_e$ is in  itself a minimal presentation if  and only if $\deg (g^0_d , \phi ^* (P)) + \deg (\xi^0_e, \phi ^* (Q))\le n$ for  any $P$, $Q \in \P ^1$.
\label{coro4}
\end{Corollary}
\begin{proof} Assume $\cl \simeq \ck _X-g^0_{d} + \xi^0_e$ with $\deg (g^0_d , \phi ^* (P)) + \deg (\xi^0_e, \phi ^* (Q))\le n$ for  any $P$, $Q \in \P ^1$.  According to Theorem  (2.6)  in \cite {AC}, a general $n$-gonal curve $X$ has a unique $g^1_n$ and any globally generated line bundle $\cm$ on $X$ with $\deg \cm \le \frac{g+1}{2}$ is composed with the $n$-fold covering morphism associated to the $g^1_n$. Hence   Theorem  \ref{mainthm} 	implies  that $\cl \simeq \ck _X-g^0_{d} + \xi^0_e$ is a minimal presentation  for $d+e \le \frac{g+3}{2}$. The converse trivially comes from Proposition \ref{thmgon}.
\end{proof}

Specifically,  we obtain the following  for a simple covering $\phi : X  \to Y$, since  $\mu \ge  \frac{g(X)-ng(Y)}{n-1} +1$ by the Castelnuovo-Severi inequality.
\begin{Corollary}  Let a smooth curve $X$ genus $g \ge 2$ admit  a simple $n$-fold covering morphism $\phi : X\to Y$ for a smooth curve  $Y$ of genus $\gamma$. If  $g^0_d$ and $\xi^0_e$ on $X$ with $d+e \le  [\frac{g-n\gamma}{n-1}] +1$ satisfy that $(g^0_d,\xi^0_e) =0$ 
and $\deg (g^0_d , \phi ^* (P)) + \deg (\xi^0_e, \phi ^* (Q))\le n$ for  any $P$, $Q \in Y$, then we obtain a nonspecial line bundle  $\cl \simeq \ck _X-g^0_{d} + \xi^0_e$  which is in
  itself a minimal presentation. Specifically, for a  double covering case  the presentation
  $\cl \simeq \ck _X -g^0_{d} + \xi^0_e$    with $d+e \le g- 2\gamma  +1$   is    minimal   if $\deg (g^0_d  , \phi ^* (Q))\le 1$ and $\deg (\xi^0_e , \phi ^* (Q))\le 1$ for any $Q \in Y$.
 \label{coro5}
\end{Corollary}
\begin{Corollary} Let  $X$ be an $n$-fold covering of $\P^1$ via $\phi : X \to \P ^1$ and $\mu$ be the same as in Theorem \ref{mainthm}. Assume that  $\xi ^0_{e+r}$ satisfies $\deg(\xi ^0_{e+r}, \phi ^* (Q)) \le 1$ and $\phi (P_1) \neq \phi (P_2 )$ for any $Q\in \P ^1$ and  $P_1 +P_2 \le \xi ^0_{e+r}$.
 If $\mbox{max} \{rn, rn-r+e \} < \mu$ then $\cl\simeq   \ck _X -rg^1_n +\xi ^0_{e+r}$ is a nonspecial line bundle minimally presented by $\cl\simeq   \ck _X -g^0_{rn-r} +\xi ^0_{e}$, where $g^0_{rn-r} := rg^1_n (-\xi ^0_{r})$, $\xi ^0_{e} :=\xi ^0_{e+r}-\xi ^0_{r}$ for some $\xi ^0_{r}\le \xi ^0_{e+r}$.
\label{corgonal}
\end{Corollary}
\begin{proof} Note that  $|rg^1_n |=g^r_{rn}$ for $rn < \mu$ and thus $h^0 (X, \co _X (rg^1_n (-\xi ^0_{r})) =1$  for $\xi ^0_{r}\le \xi ^0_{e+r}$ due to the condition that $\deg(\xi ^0_{e+r}, \phi ^* (Q)) \le 1$ for any $Q\in \P ^1$. Thus   we can set $g^0_{rn-r} := rg^1_n (-\xi ^0_{r})$, which satisfies that $(g^0_{rn-r} , \xi ^0_{e})=0$ for $\xi ^0_{e} :=\xi ^0_{e+r}-\xi ^0_{r}$. Accordingly,  $\cl\simeq   \ck _X -rg^1_n +\xi ^0_{e+r}$ admits a well defined presentation $\cl \simeq   \ck _X -g^0_{rn-r} +\xi ^0_{e}$.  Its minimality comes from Theorem \ref{mainthm}, since  $\deg (g^0_{rn-r} , \phi ^* (Q)) \le n-1$, $\deg(\xi ^0_{e+r}, \phi ^* (Q)) \le 1$ and $\phi (P_1) \neq \phi (P_2 )$ for any $Q\in \P ^1$ and  $P_1 +P_2 \le \xi ^0_{e+r}$.
\end{proof}

\begin{Remark} Let $X$ be an $n$-gonal curve of genus $g$ via $\mu : X \to \P ^1$ and $\mu$ be the same as in Theorem \ref{mainthm}. Assume that $\mbox{max} \{rn, rn-kr+e \} < \mu$. Choose a $\xi ^0_{e+kr}$ with $k\ge 1$ such that $\deg(\xi ^0_{e+kr}, \phi ^* (Q_i)) =k$ for distinct $Q_1, ..., Q_r \in \P ^1$  and $\deg (\xi ^0_{e+kr}, \phi ^* (Q)) \le k$ for any $Q \neq Q_i$, $i=1,...,r$.
Let $\xi ^0_{kr}:=\sum _{i=1} ^{r} (\xi ^0_{e+kr}, \phi ^* (Q_i))$.  Then, we have $h^0 (X, \co _X (rg^1_n (-\xi ^0_{kr})) =1$ since $|rg^1_n |=g^r_{nr} =|\phi ^* (\sum _{i=1} ^r Q_i ) |$.
Set $g^0_{rn-kr} := rg^1_n (-\xi ^0_{kr})$ and $\xi ^0_{e} :=\xi ^0_{e+kr}-\xi ^0_{kr}$.  Then $\cl\simeq   \ck _X -rg^1_n +\xi ^0_{e+kr}$ is a nonspecial line bundle  minimally presented by $\cl\simeq   \ck _X -g^0_{rn-kr} +\xi ^0_{e}$ due to Theorem \ref{mainthm}. The proof is very similar to Corollary \ref{corgonal}.
\end{Remark}

Now, we consider  a minimal presentation problem for  nonspecial line bundles  on   $X$  with a simple morphism $\phi : X \to Y$ for a  smooth plane curve $Y$, since it is possible to use some theories on linear systems on  smooth plane curves. Here, $\phi$ is said to be simple if it does not factor through. In $\S 4$,  the  result on  this will be applied to investigate property $(N_p)$ of line bundles on such curves.

\begin{Theorem}[\cite{Na}, p.82]  Let $C$ be a smooth plane curve of degree $d$. Let $g^1_n$  be a linear system on C. And let $\P _k $ be the projective space parameterizing effective divisors of degree $ k$ on $\P ^2$. If $g^1_n =\P _. C - F(\P _. C)$ for  a pencil $\P$ of  $\P _k $, then $n\ge k(d-k)$, where
$F(\P _. C):= \cap \{  E | E \in \P _. C \}. $
\end{Theorem}
This theorem gives the following lemma.
\begin{Lemma}   Let $C$ be a smooth plane curve of degree $d$.  If  $\mathcal D$ is a base point free complete linear system  on $C$ with $\dim  \mathcal D \ge 1$ and  $ \deg \mathcal D \le 2d-5$, then $\mathcal D$ is equal to either $g^1_{d-1}$ or $g^2_d$. \label{lempl}
\end{Lemma}
%From the lemma we obtain a sufficient condition for the minimality of  some nonspecial line bundles on  multiple coverings of  smooth plane curves.
%Note that those curves $X$ with $\phi : X \to Y$ for a  smooth plane curve $Y$ would be next candidates which can be dealt with, since the properties of line bundles on a smooth plane curve provide some tools to investigate the minimality of presentations.  
From Lemma \ref{lempl} we obtain  conditions for the minimality of  presentations of typical  nonspecial line bundles  such as $\ck _X -\phi ^* g^2_d +\xi_{e+2}$ and  $\ck _X -\phi ^* g^1_{d-1} +\xi_{e+1}$   on a multiple covering $X$ of  a  smooth plane curve $Y$ of degree $d$ via $\phi : X \to Y$.
\begin{Theorem} 
%Let  a smooth curve $X$ of genus $g$ be a simple $n$-fold  covering  a smooth plane curve $Y$ of degree $d \ge 5$ via $\phi : X \to Y \subset  \P^2$. And let $\delta _\epsilon :=\min \{ [\frac{g-n\gamma}{n-1}] -nd+3, nd-5n+3 \} +\epsilon (n-1)$, $\epsilon = 0,1$.\\
Let  a smooth curve $X$ of genus $g$ admit  a morphism   $\phi : X \to Y \subset  \P^2$ which does not factor through for a smooth plane curve $Y$ of degree $d \ge 5$ with $g> ng(Y) + n(n-1)d$ for   $n :=\deg \phi \ge 2$. Let $\delta _\epsilon :=\min \{ [\frac{g-ng(Y)}{n-1}] -nd+3, nd-5n+3 \} +\epsilon (n-1)$  in case $n\ge 2$; $\delta _\epsilon :=  nd-5n+3$ in case $n=1$;  $\epsilon = 0,1$.\\
% Then we have the following.\\
(i) If we take  a  $\xi ^0 _{e}$ with $e\le \delta _0$ such that $\deg (\xi ^0 _{e},\phi ^* (H)) \le 2$ for any $H\in g^2_d$, then the line bundle $\cl \simeq \ck  _X  - \phi ^* (g^2_d) +\xi ^0 _{e}$ carries a natural minimal presentation $\cl \simeq \ck  _X  -g^0_{nd-2} +\xi ^0_{e-2}$, where $g^0_{nd-2}: =\phi ^* (g^2_d) - (P_1 +P_2 )$, $\xi ^0_{e-2}:= \xi ^0 _{e} -(P_1 +P_2 ) \ge 0$.\\
(ii) If we take  a  $\xi ^0 _{e}$ with $e\le \delta _1$ satisfying that  $\deg (\xi ^0 _{e},\phi ^* (H-{\tilde Q})) \le1$ for any $ H \in g^2_d$ with
 $ H\ge {\tilde Q}$ and $\deg (\xi ^0 _{e},\phi ^* (H)) \le n+1$ for any $H\in g^2_d$, then  $\cl \simeq \ck  _X  - \phi ^* (g^2_d (-{\tilde Q})) +\xi ^0 _{e}$ carries a natural minimal presentation $\cl \simeq \ck  _X  -g^0_{nd-n-1} +\xi ^0_{e-1}$, where $g^0_{nd-n-1}: =\phi ^* (g^2_d (-{\tilde Q})) - P$, $\xi ^0_{e-1}:= \xi ^0 _{e} -P \ge 0$, $P \in X$.\\
\label{planeth}
\end{Theorem}
\begin{proof} First, we verify the theorem in the case $n\ge 2$, since the theorem for $n=1$ can be shown easily through  a  similar proof. \\
(i) Note that we get $\beta =2$ due to  $\deg (\xi ^0 _{e},\phi ^* (H)) \le 2$ for any $H\in g^2_d$, where $\beta := \mbox{max} \{ \deg (\xi ^0 _e , D) | D \in g^2_d \}$.  By Theorem \ref{submin}, it suffices to show that $\dim |\phi ^* (g^2_d) +G| =2$ for any  $G\ge 0$ with $\deg G \le  e-3$.  Since $\deg (\phi ^* (g^2_d) +G) \le nd +e-3 \le \min \{ [\frac{g-ng(Y)}{n-1}]  ,n(2d-5) \}$,  the Castelnuovo-Severi inequality  implies that 
$ |\phi ^* (g^2_d) +G|$ is composed with $\phi$ and so Lemma \ref{lempl} gives $ |\phi ^* (g^2_d) +G| =\phi ^* (g^2_d) +G$, that is, $\dim  |\phi ^* (g^2_d) +G| =2$. This completes the proof of (i).

\noindent{(ii)} According to the condition that $\deg (\xi ^0 _{e},\phi ^* (H-{\tilde Q})) \le1$ for any $ H \in g^2_d$ with
 $ H\ge {\tilde Q}$, the number $\beta : = \mbox{max} \{ \deg (\xi ^0 _e , D) | D \in \phi ^* (g^2_d (-{\tilde Q})) \}$  is equal to one and thus $\cl$ admits a presentation $\cl \simeq \ck  _X  -g^0_{nd-n-1} +\xi ^0_{e-1}$, where $g^0_{nd-n-1}: =\phi ^* (g^2_d (-{\tilde Q})) - P$, $\xi ^0_{e-1}:= \xi ^0 _{e} -P \ge 0$, $P\in X$. Assume that there is another presentation $$\cl \simeq \ck  _X  - \phi ^* (g^2_d (-{\tilde Q})) +\xi ^0 _{e}\simeq \ck  _X  -h^0_t +\zeta ^0_{s}$$ with $s\le e-2$ which also means $t \le nd-n-2$.  This yields that 
$$ |\phi ^* (g^2_d (-{\tilde Q})) +\zeta ^0_{s}| =|h^0_t +\xi ^0 _{e}| =: g^\alpha _{nd-n +s}.$$
Note that  $e\le \delta _1$ gives that $nd-n+s \le nd -n +e-2\le \min \{ [\frac{g-ng(Y)}{n-1}], n(2d-5) \}$. Accordingly,  by the Castelnuvo-Severi inequality the linear system $g^\alpha _{nd-n +s}$ is composed with $\phi$, whence by  Lemma \ref{lempl} we conclude 
$$g^\alpha _{nd-n +s}= \phi ^* g^2_d+B_{s-n \ge 0} ~\mbox{ or } ~\phi ^* g^2_d (-{\tilde Q} ) +B_{s},$$ $B_{s-n},$ $B_{s}$: base loci. Here, we trivially obtain  that $B_{s-n} \le \zeta ^0_{s}$ and  $B_{s}= \zeta ^0_{s}$. 
Assume that $g^\alpha _{nd-n +s}= \phi ^* g^2_d+B_{s-n \ge 0}$. Then there is a $H \in g^2_d$ such that $\deg (\xi ^0_e, \phi ^* H) \ge nd -t \ge n+2$ due to $ |h^0_t +\xi ^0 _{e}| =\phi ^* g^2_d +B_{s-n}$ and  $t \le nd-n-2$. This is absurd since  $\deg (\xi ^0 _{e},\phi ^* (H)) \le n+1$ for any $H\in g^2_d$.  It forces that 
$$g^\alpha _{nd-n +s}=\phi ^* g^2_d (-{\tilde Q} ) +B_{s}.$$
This means that  $\deg (\xi ^0_e, \phi ^* (H-{\tilde Q})) \ge nd -n-t \ge 2$ for some $ H \in g^2_d$ with ${\tilde Q} \le H$, since $h^0 _t +\xi ^0_e \in \phi ^* g^2_d (-{\tilde Q} ) +B_{s}.$
Accordingly,  we   also meet a contradiction to $\deg (\xi ^0 _{e},\phi ^* (H-{\tilde Q})) \le 1$ for any $ H \in g^2_d$ with  $ H\ge {\tilde Q}$. As a result, the       presentation $\cl \simeq \ck  _X  -g^0_{nd-n-1} +\xi ^0_{e-1}$ is minimal.
This completes the proof of the theorem for $n\ge 2$. 

Next, consider the case of $n =1$ which means the biregularity of $\phi$ since $Y$ is nonsingular. Thus any linear system on $X$ is composed with $\phi$ and hence we can verify the theorem by  using Lemma \ref{lempl} and substituting $n=1$ in the proof of the case $n\ge 2$. Finally, we obtain the result.
\end{proof}

Note that the Riemann-Hurwitz Formula implies  $g\ge ng(Y) -n +1$ and thus the hypothesis  $g> ng(Y) + n(n-1)d$ of  Theorem \ref{planeth} is not strong in case $d>n$.

\section{Applications to  Green-Lazarsfeld's conjecture on syzygies of curves  }

Consider  a very ample line bundle $\cl$ on a smooth curve $X$ and the  homogeneous coordinate ring  $S:=Sym(H^0 (X, \cl))$  of $\varphi _{\cl} (X)$ in $\P^r:= \P H^0 (X, \cl ) $.
Then we have  a minimal free resolution of $S(X)$ as a graded $S$-module as follows:
$$0 \to E_{r-1} \to \cdots \to E_{p}\to
 E_{p-1}\to \cdots  \to E_1  \to S \to S(X) \to 0.$$
M. Green and R. Lazarsfeld have defined  property $(N_p)$ for $\cl$, which means  $E_0 =S$ and $E_i = \bigoplus^{{\beta}_{i,1}} S(-i-1)$ for all $1\le i  \le  p$(\cite{GL}, Section 3).  In their paper, they demonstrated that  property $(N_p)$  is closely related to the  Clifford index \cli (X) of  $X$ which is  an important birational numerical invariant of a smooth curve.

 They verified in \cite{GL} that  property $(N_0)$(:normal generation) holds for any very ample line bundle  $\cl$ on $X$ with  $\deg \cl \ge 2g+1 -\cli (X) -2h^1 (X, \cl )$.  The exactness of this bound has shown in \cite{GL}, \cite{Ko}, \cite{CKK}: there are very ample line bundles with $\deg \cl = 2g -\cli (X) -2h^1 (X, \cl )$ which fail to be normally generated. In \cite{GL1}, they also proved that a line bundle $\cl$ of degree $2g+p$ on a nonhyperelliptic curve $X$ satisfies $(N_p)$ if and only if $\varphi _{\cl} (X)$ has no (p+2)-secant p-planes.
In this context,  M. Green and R. Lazarsfeld   raised in \cite{GL} the following conjecture:

\begin{Conjecture}[Green-Lazarsfeld's conjecture on  $(N_p)$]
Let  $\cl$ be a very ample line bundle on a smooth curve $X$ of genus $g$ with $\deg \cl\ge 2g+1+p-2h^1(X,\cl)-\cli(X)$.
 If  $\varphi _{\cl} (X)$ has no $ (p+2)$-secant $p$-planes  then property $(N_p)$  holds for $\cl$.  \label{conjG}
\end{Conjecture}

\begin{Remark} Any spececial very ample line bundles on a general $k$-gonal curve $X$ of genus $g$ with $3\le k< [\frac{g}{2}]+2$ satisfy  Green-Lazarsfeld's conjecture on  $(N_p)$ by Theorem 2  in \cite{Ap1} and Theorem 1 in \cite{CKKw}.  In fact, Theorem 2  in \cite{Ap1} implies that a general $k$-gonal curve $X$ with $3\le k< [\frac{g}{2}]+2$ satisfies   Green's Conjecture: $K_{p,1} (X, K_X)=0 \mbox{ if and only if ~}  p\ge g-\cli (X) -1$; and Theorem 1 in \cite{CKKw} says that  if a  very ample line bundle $\cl$ on $X$ satisfies property $(N_p)$ then $\cl(-Q)$ has property $(N_{p-1})$ for any $Q\in X$.
\label{spgon}
% Recently, M. Aprodu in \cite{Ap1} showed that  a  $k$-gonal curve $X$ of genus $g$ with $3\le k< [\frac{g}{2}]+2$  such that $\dim W^1_{d+n} (X) \le n$ for all $n \le g-2k+2$ satisfies   Green's Conjecture: $$K_{p,1} (X, K_X)=0 \mbox{ if and only if ~}  p\ge g-\cli (X) -1.$$ Thus a general $k$-gonal curve $X$ with $3\le k< [\frac{g}{2}]+2$ satisfies   Green's Conjecture, whence special line bundles on  $X$ satisfy Conjecture \ref{conjG} by Theorem \ref{proj}(Theorem 1 in \cite{CKKw}) in the following.
\end{Remark}

Now,  note that  minimal  presentations of nonspecial  line bundles  give not only information on  the existence of $ (p+2)$-secant $p$-planes but also the construction of nonspecial line bundles with/without a $ (p+2)$-secant $p$-plane.   Accordingly,  a minimal  presentation can be an effective tool to observe the exactness of    Conjecture \ref{conjG}.

\begin{Definition} Let $X$ be a non-hyperelliptic curve of genus $g$.
A very ample line bundle $\cl$ on  $X$ with $\deg \cl = 2g+p-2h^1(X,\cl)-\cli(X)$ is called {\it an extremal line bundle for
 Green-Lazarsfeld's conjecture on} $(N_p)$ if
$\cl$ does not satisfy property $(N_p)$ and  $\varphi _{\cl} (X)$ has no $ (p+2)$-secant $p$-planes.  Specifically, for $p=0$ it was already defined by  {\it an extremal line bundle} in \cite{GL}.
\end{Definition}

We will demonstrate that  general  $n$-gonal curves and some simple $n$-fold coverings of smooth plane curves carry  nonspecial extremal  line bundles for Green-Lazarsfeld's conjecture on $(N_p)$.
Furthermore,  the results also show how to construct extremal line bundles for Green-Lazarsfeld's conjecture on $(N_p)$ on such curves.

Before going to our main results of this section, we consider that any line bundle $\cl$ of degree $2g +p (= 2g+p-2h^1(X,\cl)-\cli(X))$ on a hyperelliptic curve $X$ does not satisfy  property $(N_p )$, whereas a line bundle $\cl$ of that degree on a nonhyperelliptic curve does not satisfy  property $(N_p )$ if and only if $\cl$ embeds $X$ with a $(p+2)$-secant $p$-plane(see \cite{GL1}, Theorem 2). The following proposition explicitly shows that property $(N_p )$ for the line bundle $\cl$ on hyperelliptic curves is regardless of the existence of such secant planes.

\begin{Proposition} Let $X$ be a   hyperelliptic curve of genus $g\ge 2$.  Choose two  divisors $g^0_{d\ge 1}$ and $\xi ^0_{d+p+2}$ on $X$ with  $(g^0_d,  \xi ^0_{d+p+2} )=0$ for $p\ge 0$, $p+2d \le g-1$. Then   the nonspecial line bundle $\cl \simeq \ck _X -g^0_d +\xi ^0 _{d+p+2}$ embeds $X$ without a $(p+d+1)$-secant $(p+d-1)$-plane  and  does not satisfy  property  $(N_p)$.
\label{rmkhyp}
\end{Proposition}
\begin{proof}  By  Theorem 2 in  \cite{GL1} the line bundle  $\cl$  does not satisfy the property  $(N_p)$.  Corollary \ref{coro5} implies that  $\cl \simeq \ck _X -g^0_d +\xi ^0 _{p+d+2}$ is a minimal presentation for $p+2d \le g-1$, and hence $\cl$ is very ample and    embeds  $X$ with no $(p+d+1)$-secant $(p+d-1)$-planes.
\end{proof}

The following theorem shows that a general $n$-gonal curve $X$ of genus $g$ with $3 \le n \le [\frac{g-2}{2}]$ carries   numerous  nonspecial extremal line bundles $\cl$ for Green-Lazarsfeld's conjecture on $(N_p)$, that is, (1) $\deg \cl=2g+p-\cli(X)$, (2) $\cl$ is  very ample and  does not satisfy  property $( N_p)$, (3)  $\varphi _{\cl} (X)$ has no $(p+2)$-secant $p$-planes.

\begin{Theorem} Let  $X$ be a general $n$-gonal curve of genus $g$ with $3 \le n \le [\frac{g-2}{2}]$. For $p \le \frac{g-1}{2} -n$, the curve $X$ carries    nonspecial extremal line bundles  for Green-Lazarsfeld's conjecture on $(N_p)$ which are given by $\cl \simeq \ck _X - g^1_n   +\xi ^0 _{ p+4}  $ for   some $\xi ^0 _{ p+4} $ with $\deg (\xi ^0_{p+4}, F) \le 1$ for any $F \in g^1_n$.
%choose a $\xi ^0 _{ p+4} $ satisfying $\deg (\xi ^0_{p+4}, F) \le 1$ for any $F \in g^1_n$.
%%And    let $p$ be a number satisfying . And let $\xi ^0 _{ p+4} $ satisfy $\deg (\xi ^0_{p+4}, F) \le 1$ for any $F \in g^1_n$.
% Then,  $\cl \simeq \ck _X - g^1_n   +\xi ^0 _{ p+4}  $ is an extremal line bundle
%for Green-Lazarsfeld's conjecture on $(N_p)$.
%%$X$ carries   a nonspecial extremal line bundles $\cl$ for Green-Lazarsfeld's conjecture on $(N_p)$, that is, (1) $\deg \cl=2g+p-\cli(X)$, (2) $\cl$ is
%% very ample and  does not satisfy  property $( N_p)$, (3)  $\varphi _{\cl} (X)$ has no $(p+2)$-secant $p$-plane.
 \label{eggon}
\end{Theorem}
\begin{proof}  
First, consider the case $p=0$. By Corollary \ref{coro4}, a line bundle $\cl \simeq \ck _X -g^1_n  + \xi ^0_4$ is minimally presented by 
$$\cl~ \simeq~ \ck_X-g^0_{n-1}+\xi^0_{3}, \   \ g^0_{n-1}:=g^1_n ( -P) , \
\xi ^0_{3} :=\xi ^0_{4} -P, \ P\le \xi ^0_{4},   $$ 
due to $\deg (\xi ^0_{4}, F) \le 1$ for any $F \in g^1_n$.
Thus $\cl$ is very ample. By the same arguments of  (2.1) Theorem in \cite{GL},  $\cl$ fails to be normally generated since $D:= \xi ^0_3 +P$ spans a line via the embedding $\varphi _{\cl}$.

Next,  consider $\cl \simeq \ck _X - g^1_n   +\xi ^0 _{ p+4}  $ with $0<p \le \frac{g-1}{2} -n$. Then  $\cl$ admits a well defined presentation
$$\cl~ \simeq~ \ck_X-g^0_{n-1}+\xi^0_{p+3}, \   \ g^0_{n-1}:=g^1_n ( -P) , \ 
\xi ^0_{p+3} :=\xi ^0_{p+4} -P,   \ P \le  \xi ^0_{p+4}$$ 
since we have $(g^0_{n-1} , \xi^0_{p+3})=0$ by the condition $\deg (\xi ^0_{p+4}, F) \le 1$ for any $F \in g^1_n$.  
Corollary \ref{coro4}
implies the minimality of the presentation, since  $p \le \frac{g-1}{2} -n$,  $\deg (\xi ^0_{p+4}, F) \le 1$ for any $F \in g^1_n$.  
Consequently,  $\cl$ is very ample and  embeds $X$ with no  $ (p+2)$-secant $p$-planes.
Note that we have $$\deg \cl =2g+p-\cli(X),$$ since the  Clifford index of a general $n$-gonal curve is equal to $n-2$(see \cite{B}, \cite{KeK}).

Suppose that $\cl$ satisfies   property $(N_p )$. 
According to  Theorem 1 in \cite{CKKw}(see Remark \ref{spgon}),  the line bundle
$\cm~ \simeq~ \ck_X-g^0_{n-1}+\xi^0_{3}$  is normally generated, which cannot occur.
% Set $D:= \xi ^0_3 +P$. Then,  we obtain $\dim  \langle D \rangle _{\cm} =1$ by  the  Riemann-Roch Theorem and the condition $g^0_{n-1} = g^1_n (-P) $. This implies that  $D$ fails to impose independent conditions on quadrics in $\P   : =\P H^0 (X, \cm ) $. By the exact sequence
%$$0\to\mathcal I_{D}(2)\to\mathcal I_{\mathbb P}(2)\to \mathcal O_{D}(2)\to 0 , $$ we obtain
%$h^1 (\P , {\mathcal I}_{D} (2))\neq 0$, whence  $h^1 (\P , {\mathcal I}_{X} (2))
%\neq 0$ due to  $$0\to\mathcal I_{X/\mathbb P}(2)\to\mathcal I_{D/\mathbb P}(2)\to
% \mathcal I_{D/X}(2)\to 0 , $$ since  $h^1 (X, \cm ^2 (-D))=0$.
%This is contrary to the normal generation of $\cm$. 
Thus $\cl$ does not  satisfy   property $(N_p)$. As a consequence, $\cl$ is an extremal line bundle for
 Green-Lazarsfeld's conjecture on $(N_p)$. This completes the proof of the theorem.
\end{proof}

In addition, we also want to show  the existence of   a nonspecial extremal line bundle for Green-Lazarsfeld's conjecture on $(N_p)$ on a simple multiple covering of a smooth plane curve. 
%Here, a simple covering means that its covering morphism does not factor through. 
  To do this,  we have to calculate the Clifford index of such curves.  Accordingly,  we examine the Clifford index of line bundles on multiple coverings. In the following, a line bundle $\cm$ on a multiple covering $\phi : X\to  Y$ is said to be composed with $\phi$ if $\varphi _{\cm}$ factors through $\phi$, equivalently, $\cm \simeq \phi ^* \cn$ and $h^0 (X, \cm) =h^0 (Y, \cn)$. 
%need  some lemmas regarding the Clifford index of line bundles on  simple multiple coverings, since the Green-Lazarsfeld Conjecture on $(N_p)$ is closely related to  the Clifford index of line bundles.
\begin{Lemma}
Assume that a smooth curve $X$ of genus $g$  admits
  a simple  $n$-fold covering morphism $\phi :X \rightarrow Y$
 for a smooth curve $Y$ of genus $\gamma$ with $g >n\gamma$. Let
$\cm$ be a globally generated line bundle on $X$ with  $\deg \cm \le g-1$ and $h^0 (X, \cm ) \ge 2$. Then
$\cm$ is composed with the morphism $\phi$ ~if $$ \cli (\cm
)<\min \{ ~ \frac{g-n\gamma}{n-1} -3, ~ \frac{2(2n+\mu -3)}{(2n+\mu
-1)^2}g-1~\},$$ where
$\mu:=[\frac {2n(n-1)\gamma}{g-n\gamma}]$.
\label{Prop:prop2.9}
\end{Lemma}

\begin{proof}
By Lemma 5 in  \cite{K1}, if we show
$$ \cli (\cm
)<\min \{ ~ \frac{g-n\gamma}{n-1} -3, ~ \frac{2(2n+\mu -3)}{(2n+\mu
-1)^2}g-1, ~\frac{\deg ~\ck _X \otimes \cm ^{-1}}{3}~\},$$
then $\cm$ is composed with the simple  $n$-fold covering morphism $\phi$.
The condition  $\deg \cm \le g-1$ gives
$$\frac{\deg ~\ck _X \otimes \cm ^{-1}}{3} \ge \frac{g}{3} -1 \ge \frac{2(2n+\mu -3)}{(2n+\mu -1)^2}g-1,$$
 whence the result follows.
\end{proof}

Applying  Lemma \ref{Prop:prop2.9} to a simple multiple covering $X$ of a smooth plane curve, we not only find  the Clifford index of $X$ but also  characterize the line bundles computing the Clifford index of $X$.

\begin{Proposition}
Assume that a smooth curve $X$ of genus $g$  admits a  simple $n$-fold covering morphism  $\phi : X \to Y$  for a smooth plane curve $Y$  of  degree $d$ with $d > \frac{4n^2 +2 }{3}$ and $g\ge ng(Y) +n(n-1)d +2n^2 (n-1)$. Let $\cm$ be a line bundle computing the Clifford index of $X$ with $\deg \cm \le g-1$. Then
\begin{eqnarray*}
\cm &\simeq &\phi ^* \ce (-Q),~\ Q \in Y ~\ \mbox{ in case }~ n\ge 3\\
\cm &\simeq &\phi ^* \ce ~ \mbox{ or } ~ \cl \simeq \phi ^* \ce (-Q), ~\ Q \in Y \ ~\mbox{  in case  }~ n=2,
\end{eqnarray*}
where $\ce := \co _Y (1)$.
Specifically, we have $\cli (X) = nd-n-2.$
\label{clif}\end{Proposition}
\begin{proof} Assume that $\cm =\phi ^* \cn$ on $X$ is composed with $\phi$. 
Since  $\cli (\cn ) >d-4$ for a line bundle $\cn$ on $Y$ with $\cn
\neq \ce $, we have $\cli (\cm )
> nd-4$  in case $\cn \neq \ce$ or $\ce(-Q)$. On the
one hand, $$\cli (\phi ^* \ce )= nd -4 ~ \mbox{ and }~ \cli (\phi
^*\ce(-Q)) =nd -n-2 $$ for $ Q\in Y $.
Thus Lemma \ref{Prop:prop2.9} implies the theorem  if we verify the following claim, where $\gamma :=g(Y)$ and $\mu:=[\frac {2n(n-1)\gamma}{g-n\gamma}]$.

{\bf Claim:} $ nd-n-2 <\min \{ ~ \frac{g-n\gamma }{n-1} -3, ~ \frac{2(2n+\mu -3)}{(2n+\mu
-1)^2}g-1~\}.$

{\bf Proof of Claim.} Since the condition $g\ge n\gamma  +n(n-1)d +2n^2 (n-1)$ gives $ nd-n-2 < \frac{g-n\gamma }{n-1} -3$, it  suffices to show that $ nd-n-2 < \frac{2(2n+\mu -3)}{(2n+\mu -1)^2}g-1$. To prove this, we   note the inequality
$$\mu \le d- 2n-1.$$
This  is also  given by   $g\ge n\gamma  +n(n-1)d +2n^2 (n-1)$  and $d > \frac{4n^2 +2 }{3}$ which imply   $$ 2\gamma \frac{n(n-1)}{ g-n\gamma}  -(d-2n) \le \frac{(d-1)(d-2)}{d+2n}  -(d-2n)\le  \frac{-3d +4n^2 +2}{d+2n} <0,$$
for  $\gamma =\frac{(d-1)(d-2)}{2}$.

Now, we will divide the proof  into the following three cases:

\vspace{0.2cm}
\centerline{(1) $n=2$, ~~\  \ (2) $n\ge 3$  and $\mu >0$, ~~\  \ (3) $n\ge 3$  and $\mu =0$.}

\vspace{0.2cm}
\noindent{(1)} Assume that $n=2$. In case  $\mu =0$, the inequality  $2d-4  <\frac{2(\mu +1)}{(\mu +3)^2}g-1$ trivially comes from  $g\ge n\gamma  +n(n-1)d +2n^2 (n-1)=d^2 -d+10$. Thus we assume $\mu \ge 1$.  According to the inequality $\mu \le d-2n-1=d-5$ and the condition $g\ge d^2 -d+10$,   we obtain
 \begin{eqnarray*}
& &\{\frac{2(\mu +1)}{(\mu +3)^2}g-1 \} -\{ 2d-4\}\\
& &\ge \frac{1}{(\mu +3)^2} \{ (\mu +1)(2d^2 -2d+20) -(2d-3)(\mu +3)^2 \} \\
& &=\frac{1}{(\mu +3)^2} \{ ( -(2d-3)\mu ^2 +(2d-3)(d-5)\mu ) -(d-23)\mu +2d^2 -20d+47\} \\
& &\ge \frac{1}{(\mu +3)^2} \{ -(d-23)\mu +2d^2 -20d+47\} \\
& &= \frac{1}{(\mu +3)^2} \{ -(d-23)\mu +d(d-5) +d^2 -15d+47 \}\\
& &\ge \frac{1}{(\mu +3)^2} \{ -(d-23)\mu +d\mu +d^2 -15d+47 \} >0.
\end{eqnarray*}
This proves the claim for $n=2.$

\noindent{(2)} Assume that $n\ge 3$  and $\mu >0$. Since  $g\ge n\gamma +n(n-1)d +2n^2 (n-1)$ and $2\gamma =(d-1)(d-2)$, we get
$$2g> n(d^2 -3d) +2n(n-1)d = nd(d+2n-5)\ge nd(d+1).$$
This gives
\begin{eqnarray*}
& &\{\frac{2(2n+\mu -3)}{(2n+\mu -1)^2}g-1 \} -\{ nd-n-2\}\\
& &> \frac{1}{(2n+\mu -1)^2}\{ (2n+\mu -3)nd(d + 1) -(nd-n-1) (2n +\mu -1)^2 \}\\
& &>\frac{nd}{(2n+\mu -1)^2}\{ (2n+\mu -3)(d+1) - (2n +\mu -1)^2 \}\\
& &\ge \frac{nd}{(2n+\mu -1)^2}\{ 2n +\mu -7\} \ge 0,
\end{eqnarray*}
for $\mu >0$, $\mu \le d-2n-1$. Thus  the claim  is verified in
case $n=3$  and $\mu >0$.

\noindent{(3)} Assume that $n\ge 3$  and $\mu =0$.  Then we have
\begin{eqnarray*}
& &\{\frac{2(2n+\mu -3)}{(2n+\mu -1)^2}g-1 \} -\{ nd-n-2\}\\
& &=\frac{2(2n -3)}{(2n -1)^2}g -nd +n +1> \frac{n(2n -3)}{(2n -1)^2} (d-1)(d-2)  -nd \\
& &\ge n \{ \frac{4n^2(2n -3)}{3(2n -1)^2} (d-2) -d\} \ge n \{ \frac{4}{3} (d-2) -d\}  >0,
\end{eqnarray*}
since $2g> 2n\gamma = n(d-1)(d-2)$, $d-1 \ge \frac{4n^2}{3}$,  and
$\frac{n^2(2n -3)}{(2n -1)^2} \ge 1 $ for $n\ge 3$. This proves
the claim  in case $n \ge 3$ and  $\mu =0$. Thus we complete the
proof of the theorem.
\end{proof}

From this we obtain the following theorem which   demonstrates   the sharpness of  Conjecture \ref{conjG} on  simple multiple coverings of  smooth plane curves. As mentioned in $\S$3, the condition $g\ge ng(Y) +2n(n-1)d$ in the theorem is not so strong  for $d \gg n$, since any multiple covering $X$ of a smooth curve $Y$  satisfies the inequality $g(X)\ge ng(Y) -n$ by the Riemann-Hurwitz Formula.
\begin{Theorem}
Let a smooth curve $X$ of genus $g$ be a simple $n$-fold   covering of a smooth plane curve $Y$  of degree $d$ with $d > \frac{4n^2 +2 }{3}$ and $g\ge ng(Y) +2n(n-1)d$. For any $p \le nd -4n -2$,  the curve $X$ carries  extremal line bundles  for Green-Lazarsfeld's conjecture on $(N_p)$ which are given by  $\cl \simeq \ck  _X  -\phi ^* (g^2_d (-\tilde Q)) +\xi ^0_{p+4}$ for some $\tilde Q \in Y$ and $\xi ^0_{p+4} \in X^{(p+4)}$ satisfying  that 
 $\deg (\phi ^* ( H-\tilde Q), \xi ^0_{p+4} )\le 1$  for any  $H \in g^2_d$ with $\tilde Q\le  H$ and $\deg (\phi ^* (H), \xi ^0_{p+4}) \le n+1$ for any $H \in g^2_d$.

%Assume that  $\cl \simeq \ck  _X  -\phi ^* (g^2_d (-\tilde Q)) +\xi ^0_{p+4}$ for some $\tilde Q \in Y$ and $\xi ^0_{p+4} \in X^{(p+4)}$  with $p \le nd -4n -2$ such that 
% $\deg (\phi ^* ( H-\tilde Q), \xi ^0_{p+4} )\le 1$  for any  $H \in g^2_d$ with $\tilde Q\le  H$ and $\deg (\phi ^* (H), \xi ^0_{p+4}) \le n+1$ for any $H \in g^2_d$. Then $\cl$ is an extremal line bundle for Green-Lazarsfeld's conjecture on $(N_p)$.
\label{plNp}
\end{Theorem}

 \begin{proof} 
Assume that  $\cl \simeq \ck  _X  -\phi ^* (g^2_d (-\tilde Q)) +\xi ^0_{p+4}$ for some $\tilde Q \in Y$ and $\xi ^0_{p+4} \in X^{(p+4)}$ satisfying the conditions of the theorem.
%Take a $\xi ^0_{p+4}$ of degree $p+4$ such that  $\deg (\phi ^* ( H-\tilde Q), \xi ^0_{p+4} )\le 1$  for any  $H \in g^2_d$ with $\tilde Q\le  H$. 
%Let  $\cl \simeq \ck _X -\phi ^ *g^1_{d-1} + \xi ^0_{p+4}$. 
Note that   the hypotheses $d > \frac{4n^2 +2 }{3}$ and $g\ge ng(Y) +2n(n-1)d$  imply both $g\ge ng(Y) +n(n-1)d +2n^2 (n-1)$ and $[\frac{g-ng(Y)}{n-1}] -nd +3 \ge nd -5n+3$. Hence, by Proposition \ref{clif}  and Theorem \ref{planeth}, $\cli (X)$ is equal to $nd -n-2$ and 
 $\cl$ is minimally presented by $ \ck  _X -g^0 _{nd-n-1} +\xi ^0 _{p+3} $, where $g^0 _{nd-n-1} := \phi ^* (g^2_d (-\tilde Q)) -P$, $\xi ^0 _{p+3} := \xi ^0_{p+4}- P$, $P\le \xi ^0_{p+4}$. 
It means that  $\cl$ is a very ample line bundle with $\deg \cl = 2g+p-2h^1(X,\cl)-\cli(X)$ and  the curve $\varphi _{\cl} (X)$ has no $(p+2)$-secant $p$-planes. 
%Thus the proof is completed if we show the failure of  property $(N_p )$ for $\cl$ since $\deg \cl = 2g+p-2h^1(X,\cl)-\cli(X)$ due to Proposition \ref{clif}.

Assume that property $(N_p )$ holds for $\cl$.  Then the line bundle $\cm$ given by
$$\cm~ \simeq~ \ck_X-g^0_{nd-n-1}+\xi^0_{3}, \quad \xi ^0_{3} \le\xi ^0_{p+3} , $$  is normally generated by  Theorem 1 in \cite{CKKw}. 
However,  it cannot occur by the same reason in the proof of Theorem \ref{eggon},  since we have $h^0 (X, \cm)-h^0 (X, \cm (-(P +\xi ^0_{3})))=2 $  by $g^0 _{nd-n-1} := \phi ^* (g^2_d (-\tilde Q)) -P$.  Thus 
%However, the line bundle $\cm$ cannot be normally generated as in the proof . 
 $\cl$ does not satisfy  property $(N_p)$.  Consequently,  $\cl$ is an extremal line bundle for Green-Lazarsfeld's conjecture on $(N_p)$.
\end{proof}
\begin{Remark} 
Let $X$ admit a morphism $\phi : X \to Y \subset \P^2$ for a smooth plane 
curve $Y$ of degree $d$ with $n: =\deg \phi \le 2$. Assume that $0\le p \le nd-5n -3$ and $g \ge 2g(Y) +2d+8$ for $n=2$.  Then, $X$ carries  another type of  extremal line bundles for Green-Lazarsfeld's conjecture on $(N_p)$ as follows. 
Choose a  $\xi ^0_{p+6}$ on $X$ satisfying (i) the points of  $\xi ^0_{p+6}$ are distinct and  map to  distinct points of Y, (ii) there is a  $\xi ^0_{6 } \le \xi ^0_{p+6}$ such that  points of $\phi (\xi ^0_{6 })$   lie on a conic but has no four collinear.   Then, $\cl \simeq \ck _X -\phi^* g^2_d +\xi ^0_{p+6}$ is 
 an extremal line bundle  for Green-Lazarsfeld's conjecture on $(N_p)$. 
This can be shown by the same way as Theorem \ref{plNp};  (1) $\varphi _{\cl} (X)$ has no $(p+2)$-secant $p$-plane since  Theorem \ref{planeth} implies that 
$\cl$ is minimally presented by $\ck _X -g^0_{nd-2} + \xi ^0_{p+4}$ with $g^0_{nd-2}: =\phi ^* (g^2_d) - (P_1 +P_2 )$ and $\xi ^0_{p+4}:= \xi ^0 _{p+6} -(P_1 +P_2 ) \ge 0$, (2) $\cl$ does not satisfy property $(N_p)$ by  Theorem 1 in \cite{CKKw}  since      $\cl (-\xi ^0 _p ) \simeq \ck _X -\phi^* g^2_d +\xi ^0_{6}$ fails to be normally generated  as (2.5) Remark in \cite {GL}, where $\xi ^0 _p: =\xi ^0_{p+6} -\xi ^0_{6}$.
\end{Remark}
%
% , where  (1) the points of  $\xi ^0_{p+6}$ are distinct and  map to  distinct points of Y, (2) there is a  $\xi ^0_{6 } \le \xi ^0_{p+6}$ such that  points of $\phi (\xi ^0_{6 })$   lie on a conic, but with no four collinear points.  
%This can be shown by the same way as Theorem \ref{plNp};  Theorem {planeth} implies that
%$\cl$ is minimally presented by $\ck _X -g^0_{nd-2} + \xi ^0_{p+4}$ with $g^0_{nd-2}: =\phi ^* (g^2_d) - (P_1 +P_2 )$ and $\xi ^0_{e-2}:= \xi ^0 _{e} -(P_1 +P_2 ) \ge 0$, and $ \ck _X -\phi^* g^2_d +\xi ^0_{6}$ fails to be normally generated as (2.5) Remark in \cite {GL}.
%\end{Remark}

%%%%%%%%%%%%%%%%%%%%%%%%%%%%%%%%%%%%%%%%%%%%%%%%%%%%%%%%%%%%
\pagestyle{myheadings} 
%%%%%%%%%%%%%%%%%%%%%%%%%%%%%%%%%%%%%%%%%%%%%%%%%%%%%%%%%%%%
\end{document}